\documentclass[final,3p,times,twoside]{elsarticle}

\usepackage{amssymb}
\usepackage{amsmath}
\usepackage{amsthm}
\usepackage{natbib}
\usepackage{amssymb}
\usepackage{mathtools}
\usepackage{multirow}
\usepackage{bm}
\usepackage{stmaryrd}
\usepackage{tabularx}
\usepackage{algorithm}
\usepackage{cancel}
\usepackage{algpseudocode}
\usepackage{tikz}
\usepackage{array}
\usepackage{textcomp}
\usepackage{multirow}
\usepackage{hyperref}
\usepackage{xcolor}
\usepackage{eso-pic}
\usepackage{lipsum}

\newtheorem{theorem}{Theorem}
\newtheorem{corollary}{Corollary}
\newtheorem{lemma}{Lemma}
\newtheorem{remark}{Remark}

\journal{Elsevier}

\begin{document}

\begin{frontmatter}

\title{Geometry-Preserving Reduced-Order Modeling via Immersed Tensor Decomposition (ITD)
}

\author[1]{Lei Zhang\fnref{eq}}
\author[2,3]{Jiachen Guo\fnref{eq}}
\author[1]{Guowei He}
\author[4]{Thomas J.R. Hughes}
\author[2,3]{Wing Kam Liu}
\fntext[eq]{These authors contributed equally to this work.}
\address[1]{School of Engineering Science, University of Chinese Academy of Sciences, Beijing 100049, China}

\address[2]{Department of Mechanical Engineering, Northwestern University, 2145 Sheridan Road, Evanston, 60201, IL, USA}

\address[3]{HIDENN-AI, Inc, 1801 Maple Ave, Evanston, 60201, IL, USA}

\address[4]{Oden Institute for Computational Engineering and Sciences, University of Texas at Austin, 201 E 24th St, Austin, 78712, TX, USA}

\begin{abstract}
Body-fitted finite-element methods deliver high-order accuracy but hinge on a clean, watertight, conforming mesh, a requirement that breaks down for the geometrically imperfect CAD assemblies, image-based volumetric data, and voxel-native designs that pervade biomedical engineering and additive manufacturing, where mesh generation has become the dominant cost of the analysis cycle. Immersed methods on regular background Cartesian grids sidestep body-fitted meshing, but classical implementations integrate over irregular cut subdomains, destroying the tensor-product structure that enables separable, reduced-order methods such as tensor decomposition. In this paper we propose the \emph{Immersed Tensor Decomposition} (ITD) framework, which couples a mesh-free geometric representation via body-fitted function with the separable C-HiDeNN-TD reduced-order solver to enable large-scale simulation directly on regular background voxel meshes. The geometry is encoded in three steps: a signed-distance function represents the boundary, a body-fitted function $\Phi$ approximates it with controllable error, and a low-rank Tucker decomposition provides model-order reduction; for a fixed grid spacing $h$, accuracy is improved by raising the approximation order of C-HiDeNN interpolation up to degree $p$ with a linear background mesh. The central contribution is an exact Dirichlet formulation that enforces the boundary condition strongly by multiplying the trial function with $\Phi$, so that $u=g$ holds by construction without any variational penalty or interface quadrature. We derive the corresponding transformed weak form, and establish an a priori error estimate for the formulation and assess it on canonical 2D/3D domains, demonstrating optimal convergence and robustness on non-Cartesian geometries discretized by regular voxel meshes.
\end{abstract}


\end{frontmatter}

\noindent\textbf{Highlights:}
\begin{itemize}

\item Geometry-Preserving ITD Framework: Enables large-scale simulation on voxel-native datasets by coupling mesh-free geometric representation with C-HiDeNN-TD solvers.
\item Exact Dirichlet Enforcement: Strong boundary condition satisfaction via body-fitted functions eliminates the need for variational penalties or complex interface quadrature.

\item Optimized Numerical Performance: Optimal $O(h^{p+1})$ convergence is established through a priori error analysis

\end{itemize}


\section{Introduction}
\label{sec:introduction}

Body-fitted finite-element methods (FEM) deliver high-order accuracy on smooth Computer-Aided Design (CAD) geometries, achieving the optimal $\mathcal{O}(h^{p+1})$ convergence rate in the $L^2$ norm for polynomials of degree $p$\cite{hughes2003finite}. This accuracy, however, is predicated on the availability of a clean, watertight, conforming mesh. In modern computational-mechanics workflows, generating such a mesh has become the dominant cost: industrial CAD assemblies routinely contain slivers, gaps, and overlapping features, and the associated geometric clean-up and mesh generation consume the majority of the analysis cycle rather than the solution itself~\cite{hughes2005iga}. The pursuit of analysis frameworks that relax the conforming-mesh requirement is therefore a central theme of contemporary numerical analysis.

In several application areas the very premise of a clean parametric CAD file no longer holds. In biomedical engineering and geophysics, geometries such as trabecular bone microstructures, soft-tissue organs, and porous rock formations are acquired by Computed Tomography (CT) or Magnetic Resonance Imaging (MRI), which produce volumetric density fields rather than boundary representations (B-reps)~\cite{schillinger2015fcm}. In additive manufacturing---particularly multi-material printing and topology optimization---the design space is intrinsically voxelated, with spatially varying material properties and lattice infills that resist a clean B-rep description\cite{saha2025efficient}. Converting such volumetric or voxel-native data into smooth NURBS surfaces is both computationally expensive and information-lossy. These trends motivate analysis frameworks that operate \emph{directly} on volumetric representations, bypassing surface reconstruction altogether.

Immersed (or embedded) methods sidestep body-fitted meshing by discretizing the governing partial differential equations on a regular Cartesian background mesh, on which the physical domain $\Omega$ is defined implicitly---most often through a level-set or signed-distance function. This idea underlies the finite cell method and its higher-order and image-based extensions~\cite{parvizian2007fcm,duster2008fcm,schillinger2015fcm}, among other fictitious-domain techniques. On the immersed Dirichlet boundary $\partial\Omega$, the constraint $u = g$ can be imposed either weakly, through Nitsche-type penalization of the boundary residual~\cite{nitsche1971,embar2010,burman2010,burman2012}, or \emph{strongly}, through trial-function multiplication. The latter is exemplified by the $\phi$-FEM approach~\cite{duprez2020phi,duprez2023phifem}, in which the trial function is constructed to vanish on $\partial\Omega$ by multiplication with a level-set function, so that the boundary condition is satisfied exactly by construction rather than enforced variationally.

Classical implementations of these strategies share a common difficulty: they require numerical integration over geometrically irregular cut subdomains, and often along the immersed interface, typically via cut-cell or shifted-boundary quadrature~\cite{main2018shifted}. This integration complicates assembly and, crucially, breaks the tensor-product structure of the underlying Cartesian mesh---the very structure that makes regular grids attractive for large-scale computation in the first place.

Preserving the tensor-product structure of the background grid is especially valuable because it is the prerequisite for reduced-order tensor methods. Fields defined on a regular Cartesian grid admit accurate low-rank Tucker and Canonical-Polyadic (CP) representations, and the trial space itself can be built in a separable format such as Proper Generalized Decomposition (PGD)~\cite{ammar2006pgd,chinesta2010pgd} or the more recent Convolution Hierarchical Deep-learning Neural Network with tensor decomposition (C-HiDeNN-TD)~\cite{lu2023chidenn,guo2024chidenntd,guo2024taps}, which together turn $d$-dimensional problems into a sequence of one-dimensional ones and break the curse of dimensionality. Such separable solvers, however, require every operator in the weak form to retain a tensor-product form---precisely what cut-cell quadrature destroys.

A natural route to reconcile the two is to represent the geometry by a \emph{smooth} implicit field~\cite{liLowengrub2009diffuse,burger2017diffuse} and to enforce the Dirichlet condition by multiplying the trial space with a body-fitted mask that vanishes on $\partial\Omega$ and itself admits a low-rank tensor representation. All integrals then remain volume integrals over the regular grid and reduce to tensor-product Gauss quadrature, with no cut-cell bookkeeping. Yet, despite this apparent compatibility, a systematic integration of strong, body-fitted immersed enforcement with separable tensor-decomposition solvers has not, to our knowledge, been reported. Several practical issues remain open, including the robust evaluation of the mask and its derivatives from voxelated signed-distance data and the control of conditioning in the fictitious (void) region.

In this paper we propose the \emph{Immersed Tensor Decomposition} (ITD) framework, which couples a body-fitted function with the C-HiDeNN-TD reduced-order solver~\cite{lu2023chidenn,guo2024chidenntd,guo2024taps} to enable large-scale simulation on non-Cartesian domains using regular voxel meshes. The specific contributions are as follows.
\begin{itemize}
    \item \textbf{An exact Dirichlet immersed C-HiDeNN-TD formulation.} The trial function is multiplied by a body-fitted function $\Phi$ (BFF) based on the signed distance function (SDF) that vanishes on $\partial\Omega$, yielding \emph{strong}, exact enforcement of the Dirichlet condition without any variational penalty or interface quadrature. We derive the corresponding transformed weak form, in which a key algebraic cancellation reduces the geometric coefficients that arise naively (five in 2D, six in 3D) to only three Tucker-decomposable fields, $\Phi$, $\Phi^2$, and $\Phi\,\nabla^2\Phi$. This reduction is what makes the formulation tractable in a separable tensor-decomposition setting.
    \item \textbf{An a priori error analysis.} We establish an a priori error estimate for the formulation (Theorem~\ref{thm:error}), characterizing its convergence behavior on immersed domains.
    \item \textbf{Numerical study.} The formulation is evaluated on canonical 2D circular and annular domains, a 3D spherical domain, and complex starfish- and gear-shaped domains, demonstrating optimal convergence and robustness on non-Cartesian geometries discretized by regular voxel meshes.
\end{itemize}

The remainder of the paper is organized as follows. Section~\ref{sec:exact} develops the exact Dirichlet immersed C-HiDeNN-TD, which constitutes the main contribution of this work. Section~\ref{sec:numerical} presents numerical examples, Section~\ref{sec:discussions} contains a discussion, and Section~\ref{sec:CONCLUSION} outlines conclusions and future work. The proof of Theorem~\ref{thm:error} is given in Appendix~\ref{appenx:Proof}.

\section{Geometric Representation using mesh-free body-fitted functions}
\label{sec:geometric_representation}

Different geometric representations have been widely adopted across various computational domains based on their specific application requirements. In computer graphics and geometric modeling, researchers have extensively leveraged the signed distance function (SDF) as an efficient, continuous mathematical representation of shapes. By definition, an SDF calculates the shortest Euclidean distance from any spatial point $\mathbf{x}$ to the domain boundary $\partial \Omega$, with the sign designating whether the point resides inside (typically negative) or outside the domain. This implicit formulation is highly advantageous for parametric modeling in the latent space; a notable example is DeepSDF, which utilizes deep auto-decoders to learn continuous SDFs for high-fidelity 3D shape generation and completion \cite{park2019deepsdf}. 

\begin{figure}[htbp]
\centering
\includegraphics[width=0.4\textwidth]{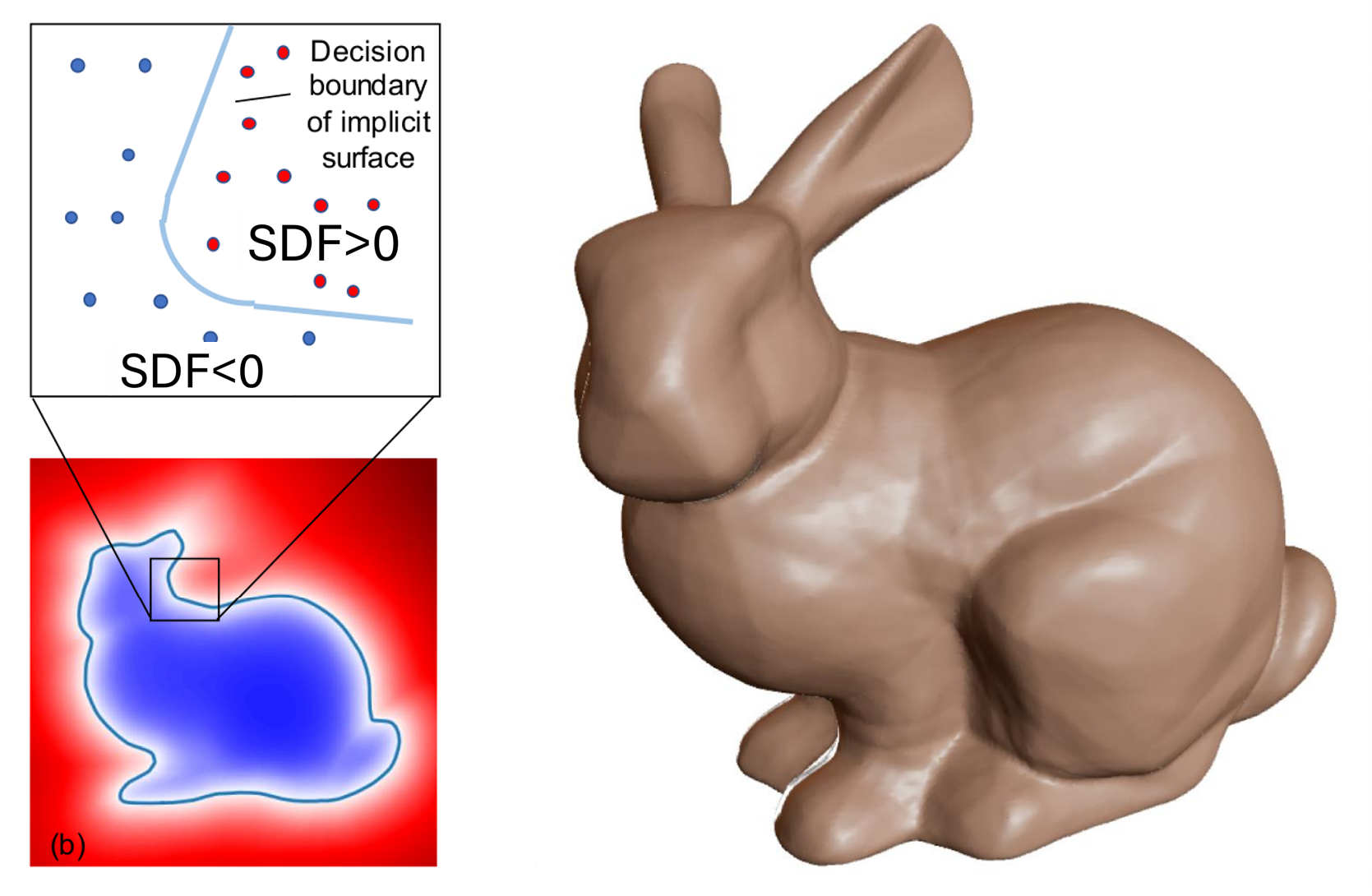}
\caption{The definition of signed distance function\cite{park2019deepsdf}}
\label{fig:sdf}
\end{figure}

Despite these useful mathematical characteristics, SDFs cannot be directly utilized as geometric descriptors for rigorous physics-based simulation, particularly in the field of Computer-Aided Engineering (CAE). In CAE, spatial fields such as material density or structural compliance possess distinct, conserved physical meanings within the domain $\Omega$. Because the SDF strictly describes relative proximity to the boundary $\partial \Omega$ without capturing interior field physics, it fundamentally lacks physical interpretation across the problem domain. 

To bridge this gap, we propose a new geometric representation termed \emph{body-fitted functions} (BFF), whose support is precisely the computational domain $\Omega$, defined by
\begin{equation}
    \Phi(\mathbf{x}) = \left\{\begin{array}{cc}
        =0, & \text{exterior} \\
        =0, & \text{on } \partial\Omega \\
        >0, & \text{interior}
    \end{array}\right.
\end{equation}
with a constraint of non-vanishing interior boundary gradient
\begin{equation}
    \lim_{\mathbf{x}\in \partial \Omega, \eta \to 0^+} |\nabla \Phi(\mathbf{x}-\eta \bm{n}) | \geq a >0,
\end{equation}
where $\bm{n}$ is the unit outward normal vector, and $a$ is a positive lower bound constant.
 The BFF thus acts as a smooth indicator of the physical domain: it is positive within the phyical domain, identically zero outside and on the boundary, and possesses a non-vanishing normal gradient at the boundary.

Specifically, we develop a formulation in which Dirichlet boundary conditions are enforced \emph{strongly} by multiplying the trial function $v^h$ with a body-fitted mask $\Phi(\mathbf{x})$ that strictly vanishes on $\partial \Omega$ there regardless of the discrete representation of $v^h$. Historically, meshfree methods and physics-informed machine learning have relied on approximate distance functions to satisfy kinematic admissibility a priori \cite{sukumar2021exact}. Our approach guarantees exact geometric enforcement, entirely eliminating the boundary-residual penalty terms, and the associated hyperparameter tuning, which are required by traditional Nitsche-type methods. This geometric enforcement is achieved at the price of an additional pre-computation effort to assemble the geometric tensors that arise in the bilinear form.

\subsection{SDF-based construction of body-fitted functions}

Importantly, the definition of the BFF does not prescribe a specific functional form: any sufficiently smooth function meeting these conditions qualifies as a body-fitted function. Different applications may call for different constructions, for instance, a polynomial level-set, a tanh-based phase field, or a data-driven neural implicit.

While the general definition admits multiple constructions, a particularly convenient route exploits the SDF, which is often already available from the geometry. Here, we adopt the following SDF-based exponential mask as the BFF:
\begin{equation}
    \Phi_\beta(\mathbf{x}) \;=\;
    \begin{cases}
        1 - \mathrm{e}^{\,\mathrm{SDF}(\mathbf{x})/\beta}, & \mathbf{x} \in \Omega,\\
        0, & \mathbf{x} \notin \Omega,
    \end{cases}
    \label{eq:Phi_def}
\end{equation}
which decays exponentially toward the boundary at a rate set by a scaling parameter $\beta$. By operating on the standard convention where $\mathrm{SDF}(\mathbf{x}) \le 0$ inside the domain, this formulation guarantees that $\Phi(\mathbf{x})$ smoothly approaches zero at the boundary while avoiding the gradient singularities often introduced by absolute distance fields. The parameter $\beta$ controls how sharply $\Phi$ transitions from 0 to approximately 1. The proposed body-fitted geometry is illustrated in Figure~\ref{fig:Phi_def}.

\begin{figure}[htbp]
\centering
\includegraphics[width=3in]{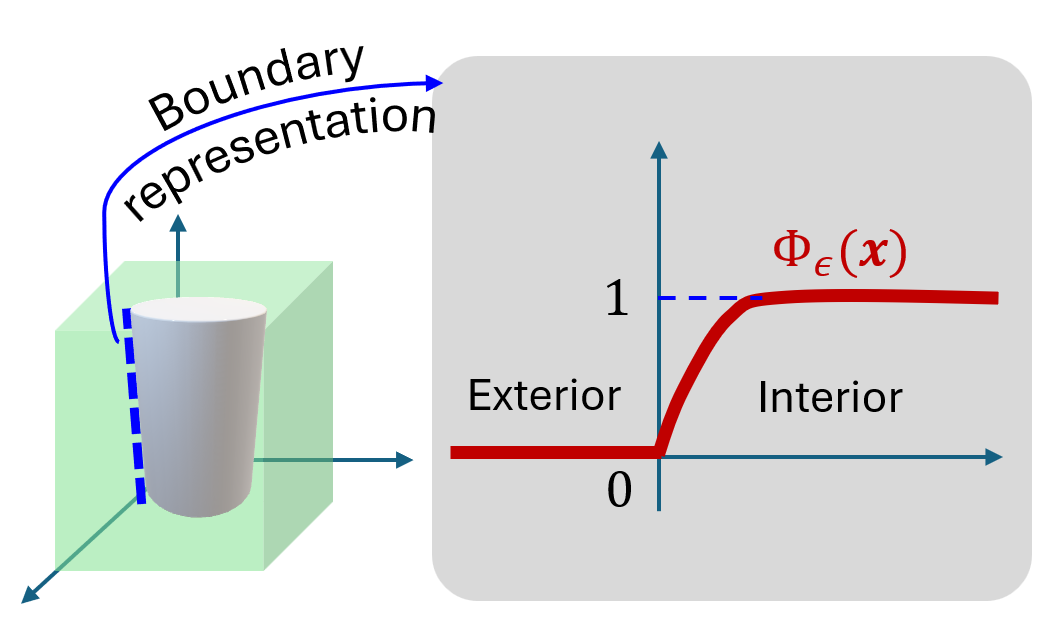}
\caption{Body-fitted function transforms SDF into a geometric descriptor that can automatically satisfy the Dirichlet boundary condition.}
\label{fig:Phi_def}
\end{figure}

\subsection{Tucker decomposition of body-fitted functions}

Building upon the proposed body-fitted geometric representation, we next seek to solve the governing partial differential equations (PDEs) through a reduced-order model (ROM) constructed in a purely \emph{data-free} manner. Traditional ROM techniques, such as Proper Orthogonal Decomposition (POD), typically require an expensive offline phase to generate high-fidelity snapshot data to form a reduced basis. To circumvent this computational bottleneck, we leverage Tucker decomposition to formulate a low-rank representation of the BFF $\Phi(\mathbf{x})$ using C-HiDeNN Tucker Decomposition as shown in Eq. \ref{eq:tucker_Phi}.
\begin{equation}
\Phi(x,y,z)\approx\; \sum_{p=1}^{P}\sum_{q=1}^{Q}\sum_{k=1}^{K} \mathcal{G}_{pqk}\,\Phi_x^{(p)}(x)\,\Phi_y^{(q)}(y)\,\Phi_z^{(k)}(z),
\label{eq:tucker_Phi}
\end{equation}
where $P,Q,K$ are the number of Tucker modes; $\Phi_{x_i}^{(p)}(x_i)$ refers to the Tucker factor in each mode $p$ and dimension $x_i$; $\mathcal{G}_{pqk}$ is the core tensor that has a reduced dimension of $P\times Q \times K$. Each factor is interpolated using C-HiDeNN shape function $\bm{\widetilde{N}}(x_i)$ \cite{guo2024taps}.

\begin{equation}
    u^{h}(x_i)=\sum_{k=1} ^ {nnode}\widetilde{N}_{k}(x_i;a_{k},s_{k},p_{k}, \mathcal{A}_k)u_{k}
\label{full_chidenn}
\end{equation}

\noindent
where $u^{h}(x_i)$ is the approximated function; $nnode$ is the total number of nodes and $k$ is the nodal index; $\widetilde{N}_{k}(x_i)$ is the C-HiDeNN shape function at the $k$-th node in dimension $x_i$; $a_k,s_k,p_k, \mathcal{A}_k$ are the node-wise hyperparameters in the C-HiDeNN shape function, which controls the smoothness of the interpolation.  It should be noted that these hyperparameters  can vary across different nodal patches since C-HiDeNN can optimize these hyperparameters like machine learning parameters, rendering an adaptable functional space without altering the number of global nodes or hidden layers. Detailed formulation of C-HiDeNN shape function can be found in \cite{lu2023chidenn,park2023convolution}. 

\begin{figure}[htbp]
\centering
\includegraphics[width=0.7\textwidth]{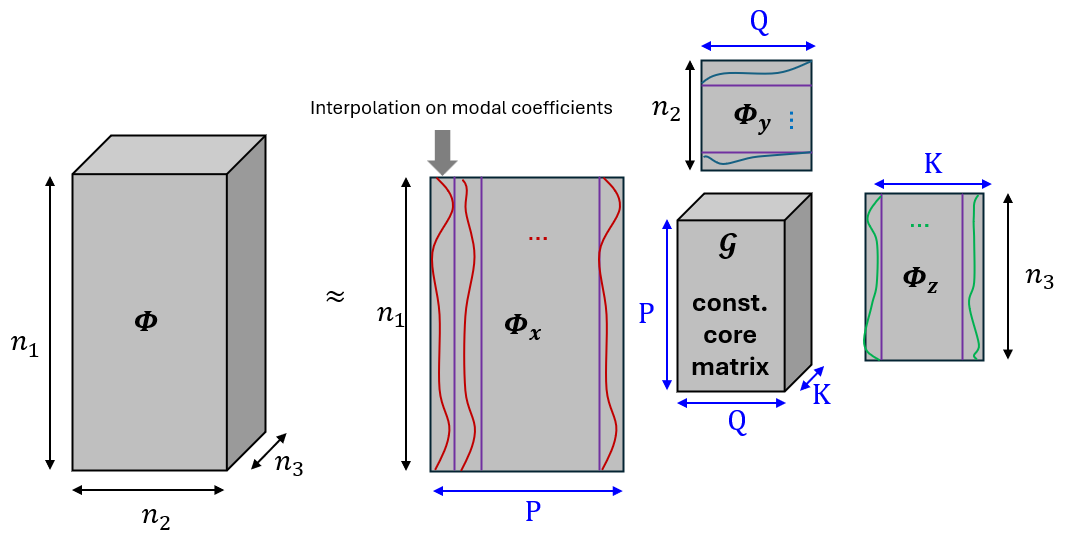}
\caption{The body-fitted function is compressed using C-HiDeNN Tucker Decomposition}
\label{fig:tucker}
\end{figure}

As shown in Fig \ref{fig:tucker}, by expressing the high-dimensional discretized space as a compact core tensor multiplied by corresponding factor matrices along each spatial mode, Tucker decomposition systematically reduces the geometric complexity. Operating directly within this low-rank tensor format allows us to evaluate and solve the governing equations in the low-rank representation. Most importantly, projecting the formulation into this latent space drastically reduces the total number of unknowns required to solve the system, yielding significant computational savings while preserving the rigorous enforcement of boundary conditions and the underlying physics.


\section{Exact Dirichlet immersed method}
\label{sec:exact}

With the body-fitted geometric representation established, we now construct the exact Dirichlet immersed formulation. The trial solution is built by multiplying the unknown field with $\Phi$, thereby satisfying the Dirichlet condition by construction. The weak form is then derived, and a key algebraic cancellation drastically reduces the number of geometric coefficients. 

We consider the following Poisson equation defined in an arbitrary domain $\Omega$:
\begin{equation}
    \nabla^2u(\mathbf{x}) + b(\mathbf{x}) = 0 \quad \text{in } \Omega,
    \label{eq:poisson_strong}
\end{equation}
which is subjected to the Dirichlet boundary condition:
\begin{equation}
    u(\mathbf{x}) =\; g(\mathbf{x}) \quad \text{on } \partial\Omega,
    \label{eq:dirichlet_bc}
\end{equation}
where $u(\mathbf{x})$ is the solution; $b(\mathbf{x})$ is the forcing term; $g(\mathbf{x})$ is the boundary condition.

\subsection{Strong Dirichlet enforcement using body-fitted function}
\label{sec:exact:strong}

We introduce a Cartesian background region $\mathcal{D}$ that covers the computational domain $\Omega$, as illustrated in Fig.~\ref{fig:background_mesh}. A structured Cartesian mesh is laid over $\mathcal{D}$, independent of the geometry of $\Omega$. The mesh nodes do not necessarily lie on $\partial\Omega$, and the elements are in general incompatible with the geometry. This incompatibility is resolved by the BFF $\Phi$, which encodes $\Omega$ exactly as defined in Section~\ref{sec:geometric_representation}. Using a regular Cartesian grid offers two principal advantages: (1) it eliminates mesh generation at essentially zero cost, since the grid is constructed by tensor-product subdivision of $\mathcal{D}$; and (2) the tensor-product structure is naturally compatible with the Tucker-decomposition-based reduced-order formulation developed in Section~\ref{sec:itd:exact}, where the separation of variables along each Cartesian direction enables significant computational savings.

\begin{figure}[htbp]
\centering
\includegraphics[width=1.5in]{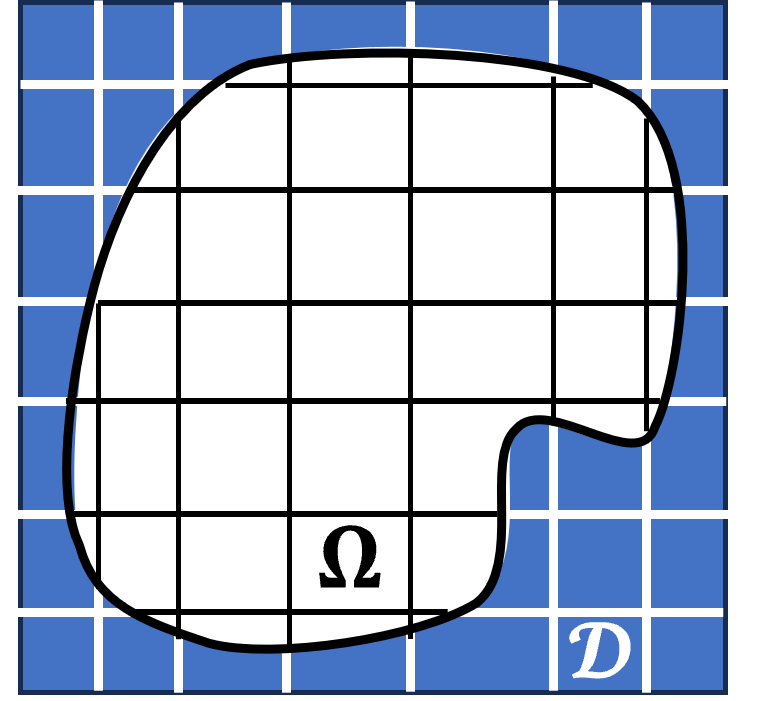}
\caption{An arbitrary domain $\Omega$ is embedded into a Cartesian background region $\mathcal{D}$ and covered by a structured Cartesian mesh.}
\label{fig:background_mesh}
\end{figure}

The trial solution defined on the Cartesian background mesh is constructed as
\begin{equation}
    u^h(\mathbf{x}) \;=\; \Phi(\mathbf{x})\,v^h(\mathbf{x}) + g^h(\mathbf{x}),
    \label{eq:strong_split}
\end{equation}
where $g^h$ is a sufficiently smooth extension of the Dirichlet boundary condition into the background domain $\mathcal{D}$, and $\Phi$ is a BFF satisfying $\Phi(\mathbf{x}) > 0$ in $\Omega$, $\Phi = 0$ on $\partial \Omega$, and $\Phi \equiv 0$ in the void $\mathcal{D}\setminus\overline{\Omega}$. Equation~\eqref{eq:strong_split} therefore enforces $u^h = g^h = g$ exactly on $\partial \Omega$, regardless of the discrete representation of $v^h$, even though the background Cartesian mesh may not conform to the geometry $\Omega$. 

\subsubsection{Weak form}

Substituting Eq. \eqref{eq:strong_split} into Eq. \eqref{eq:poisson_strong}, the auxiliary variable $v$ satisfies
\begin{equation}
    \nabla^2(\Phi v^h) \;=\; -\bigl(\nabla^2 g^h + b\bigr) \quad \text{in } \Omega,
    \label{eq:strong_v_pde}
\end{equation}
where $\nabla^2 g^h + b$ become a new source term.

Using $\Phi w^h$ as the test function (so that $\Phi w^h = 0$ on $\partial \Omega$ and the boundary contributions on $\partial \Omega$ vanish), multiplying Eq. \eqref{eq:strong_v_pde} by $\Phi w^h$ and integrating over $\mathcal{D}$ yields, after integration by parts,
\begin{equation}
    \int_{\mathcal{D}} \nabla(\Phi w^h)\cdot \nabla(\Phi v^h) \,\mathrm{d}\mathcal{D}
    \;=\; \int_{\mathcal{D}} \Phi\, w^h\,\bigl(\nabla^2 g^h + b\bigr) \,\mathrm{d}\mathcal{D}.
    \label{eq:weak_strong_pre}
\end{equation}

Expanding the bilinear form,
\begin{equation}
    \nabla(\Phi w^h)\cdot \nabla(\Phi v^h) \;=\;
    \Phi^2\,(\nabla w^h \cdot \nabla v^h)
    + \Phi\, w^h\,(\nabla v^h \cdot \nabla \Phi)
    + \Phi\, v^h\,(\nabla w^h \cdot \nabla \Phi)
    + |\nabla \Phi|^2\,w^h\, v^h.
    \label{eq:phiw_phiv_expanded}
\end{equation}
The two cross terms combine via $\Phi w^h(\nabla v^h \cdot \nabla \Phi) + \Phi v^h(\nabla w^h \cdot \nabla \Phi) = \tfrac{1}{2}\,\nabla(\Phi^2)\cdot \nabla(w^h v^h)$, and a further integration by parts, together with cancellation with the $|\nabla \Phi|^2 w^h v^h$ term,  yields
\begin{equation}
    \int_{\mathcal{D}} \Phi^2\,(\nabla w^h \cdot \nabla v^h) \,\mathrm{d}\mathcal{D}
    \;-\; \int_{\mathcal{D}} (\Phi\,\nabla^2 \Phi)\, w^h\, v^h \,\mathrm{d}\mathcal{D}
    \;=\; \int_{\mathcal{D}} \Phi\, w^h\,\bigl(\nabla^2 g^h + b\bigr) \,\mathrm{d}\mathcal{D}.
    \label{eq:weak_strong_unstabilized}
\end{equation}

The transformed weak form~\eqref{eq:weak_strong_unstabilized} is algebraically equivalent to the standard weak form obtained by substituting~\eqref{eq:strong_split} into the Poisson equation; both describe exactly the same continuum problem. The difference lies in the structure of the bilinear form. In a naive expansion of the bilinear form~\eqref{eq:phiw_phiv_expanded}, five (2D) or six (3D) distinct geometric coefficient fields appear: $\Phi^2$, $\Phi\,\nabla\Phi$ (two or three vector components), $|\nabla\Phi|^2$, and $\Phi$ on the right-hand side, each requiring an independent low-rank Tucker approximation in the reduced-order setting. After the algebraic cancellation, the bilinear form involves only three geometric fields, and $v^h$ couples to the geometry through only two bilinear integrals: a weighted stiffness term ($\Phi^2 \nabla w^h \cdot \nabla v^h$) and a weighted mass-like term ($\Phi\,\nabla^2\Phi\, w^h v^h$). The right-hand side involves a single geometric coefficient $\Phi$. Consequently, in the subsequent tensor-decomposition ROM (Section~\ref{sec:itd:exact}), only three coefficient fields must be Tucker-approximated, directly reducing the offline pre-computation and the per-mode assembly cost by nearly a half.

\subsubsection{Stabilization in the void}

In the void $\mathcal{D}\setminus\overline{\Omega}$, both $\Phi$ and $\Phi^2$ vanish, so the bilinear form in Eq. \eqref{eq:weak_strong_unstabilized} is degenerate and yields a singular stiffness matrix. To stabilize the discrete system, we add an isotropic stiffness in $\mathcal{D}\setminus\overline{\Omega}$:
\begin{equation}
    \int_{\mathcal{D}} \bigl(\Phi^2 + \gamma(\mathbf{x})\bigr)\,(\nabla w^h \cdot \nabla v^h) \,\mathrm{d}\mathcal{D}
    \;-\; \int_{\mathcal{D}} (\Phi\,\nabla^2 \Phi)\, w^h\, v^h \,\mathrm{d}\mathcal{D}
    \;=\; \int_{\mathcal{D}} \Phi\, w^h\,\bigl(\nabla^2 g^h + b\bigr) \,\mathrm{d}\mathcal{D},
    \label{eq:weak_strong}
\end{equation}
where $\gamma(\mathbf{x}) \ll 1$ is a small stabilization function.

Compared with existing immersed methods that enforce Dirichlet conditions through boundary terms, the stabilization strategy in Eq.~\eqref{eq:weak_strong} is structurally simpler. Nitsche-type methods~\cite{nitsche1971,embar2010,burman2010,burman2012,nguyen2018diffuse}, in both their sharp-interface and diffuse-interface forms, require a consistency term, an adjoint consistency term, and a penalty term. Each of these involves the boundary measure and each demands a separate Tucker compression in the reduced-order setting. The $\phi$-FEM approach~\cite{duprez2020phi,duprez2023phifem}, while sharing the idea of trial-function multiplication with the present work, still relies on a boundary stabilization to control the normal-flux mismatch across the immersed interface, introducing additional surface-localized integrals. In contrast, the proposed formulation stabilizes the system entirely through the volumetric stiffness shift $\gamma(\mathbf{x})$ in the void. Since the Dirichlet condition is satisfied strongly by construction, no boundary residual, no penalty parameter, and no interface quadrature are needed. This has two practical consequences for tensor-decomposition reduced-order models: (1) only volume fields are Tucker-compressed, rather than multiple boundary-localized fields; (2) the global domain decomposition in TD proceeds uniformly across the entire background $\mathcal{D}$ without partitioning at the immersed interface, and the 1D sub-problems inherit only smooth volumetric coefficient fields.


\subsection{Error estimate}
\label{sec:exact:error}

The exact Dirichlet formulation \eqref{eq:weak_strong} admits the following a priori error analysis. A sketch of the proof is referred to ~\ref{appenx:Proof}. We emphasize that Eq.~\eqref{eq:Phi_def} is merely one concrete realization of the BFF concept. The theoretical framework developed in this Subsection is independent of such construction, requiring only the defining properties and sufficient regularity of $\Phi$.

To facilitate the analysis we recall two extension operators. The \emph{Stein extension}~\cite{stein1967integrales, stein1970singular} is a total extension operator $E: H^s(\Omega) \to H^s(\mathbb{R}^d)$ that extends Sobolev functions from a Lipschitz domain $\Omega$ to the whole space while preserving the $H^s$ regularity. For the error estimate we require a Stein extension $\tilde{v} \in H^{p+1}(\mathcal{D})$ that coincides with $v$ on $\Omega$ and possesses the same smoothness as the exact solution across the background domain. The \emph{harmonic extension} $\hat{v}$ of $v$ into the void $\mathcal{D}\setminus\overline{\Omega}$ is defined as the solution of $\Delta \hat{v} = 0$ in $\mathcal{D}\setminus\overline{\Omega}$ with $\hat{v} = v$ on $\partial\Omega$, and $\hat{v} = v$ in $\Omega$, whose weak form is precisely the stabilization term with the weighting coefficient $\gamma(\mathbf{x})$. While the harmonic extension exactly maintains the Laplacian of the interior solution in the void that differs from the original equation, it is generally only $C^0$ across $\partial\Omega$, whereas the Stein extension preserves the full $H^{p+1}$ regularity at the expense of not satisfying any PDE in the void. The mismatch $\|\hat{v} - \tilde{v}\|_{H^1(\mathcal{D}\setminus\overline{\Omega})}$ appearing in the error bound quantifies this regularity gap.

\begin{theorem}\label{thm:error}

Let $\mathcal{D}\subset \mathbb{R}^d$ be a bounded Lipschitz-continuous Cartesian background domain, and let $\Omega\subset \mathcal{D}$ be the physical domain with a $C^2$-smooth boundary $\partial \Omega$.
Let the exact solution be $u = \Phi v \in H^{p+1}(\Omega)$. There exist a Stein extension $\tilde{v} \in H^{p+1} (\mathcal{D})$ and a harmonic extension $\hat{v}$ such that $\tilde{v}|_{\Omega}=\hat{v}|_{\Omega} = v|_{\Omega}$. Let the exact solution outside $\Omega$ be $u=\Phi \hat{v}$.
Let $u^h = \Phi v^h + g^h$ be the discrete C-HiDeNN approximation obtained from Eq. \eqref{eq:weak_strong} with mesh size $h$, polynomial order $p$, stabilization function $\gamma(\mathbf{x}) \geq 0$. Then there exist constants $C_1, C_2, C_3>0$, independent of $h$, such that the following error estimate holds
\begin{equation} \label{eq:error_estimate}
    \| u - u^h \|_{\gamma} \;\leq\; C_1 h^{p}\, \|\tilde{v}\|_{H^{p+1}(\mathcal{D})} + C_2 \sup_{\mathbf{x}\in\Omega} \gamma(\mathbf{x}) \sup_{\eta^h \in \mathcal{V}^h\setminus\{0\}} \dfrac{\|\nabla \eta^h/\Phi\|_{L^2(\Omega)}}{\|\nabla \eta^h\|_{L^2(\Omega)}} \cdot \|u\|_{H^1(\Omega)}
    + C_3 \sqrt{\sup_{\mathbf{x}\in\mathcal{D}\setminus\Omega} \gamma(\mathbf{x})}\, \|\hat{v} - \tilde{v}\|_{H^1(\mathcal{D}\setminus\overline{\Omega})}, 
\end{equation}
where the norm $\|\cdot \|_\gamma$ is defined by
\begin{equation} 
	\| w \|_{\gamma} = \sqrt{\int_{\Omega} (\nabla w)^2 \mathrm{d} \Omega + \int_{\mathcal{D}} \gamma(\mathbf{x}) \left[ \nabla (w/\Phi) \right]^2 \mathrm{d} \mathcal{D}}.
\end{equation}

\end{theorem}

In Eq. (\ref{eq:error_estimate}), the first term represents the C-HiDeNN interpolation error. The second and third terms both arise from the stabilization: the second term accounts for the deviation from the original governing equation inside the computational domain $\Omega$, while the third term captures the smoothness issue across the boundary $\partial \Omega$ caused by the harmonic extension outside $\Omega$. 

\begin{corollary}
    Let $\gamma(\mathbf{x}) = 0$ in $\Omega$, the second term in inequality (\ref{eq:error_estimate}) vanishes, the error estimate becomes
    \begin{equation}
        \| u - u^h \|_{\gamma} \;\leq\; C_1 h^{p}\, \|\tilde{v}\|_{H^{p+1}(\mathcal{D})} 
        + C_3 \sqrt{\sup_{\mathbf{x}\in\mathcal{D}\setminus\Omega} \gamma(\mathbf{x})}\, \|\hat{v} - \tilde{v}\|_{H^1(\mathcal{D}\setminus\overline{\Omega})}.
    \end{equation}
\end{corollary}

Corollary~1 shows that setting $\gamma = 0$ inside $\Omega$ eliminates the second term of~\eqref{eq:error_estimate} entirely, leaving only the C-HiDeNN interpolation error (first term) and the extension mismatch from the void stabilization (third term). This is the most natural choice for the exact Dirichlet formulation: since the BFF already enforces the boundary condition strongly, there is no need for interior stabilization to restore coercivity, and the error is dictated solely by the approximation properties of the trial space and the regularity gap across $\partial\Omega$.

\begin{corollary}
    Let $\gamma(\mathbf{x}) \sim O(h^{2p})$, the error estimate becomes
    \begin{equation}
        \| u - u^h \|_{\gamma} \;\leq\; C_1 h^{p}\, \|\tilde{v}\|_{H^{p+1}(\mathcal{D})} + C_3 h^p \, \|\hat{v} - \tilde{v}\|_{H^1(\mathcal{D}\setminus\overline{\Omega})}.
    \end{equation}
\end{corollary}

Corollary~2 provides a practical guideline for choosing the void stabilization. When $\gamma$ scales as $\mathcal{O}(h^{2p})$, the third error term, which arises from the smoothness issue of the harmonic extension, converges at the same rate $\mathcal{O}(h^p)$ as the interpolation term. Balancing the two sources of error in this way recovers the overall $h$-convergence rate of the underlying C-HiDeNN approximation and prevents the void stabilization from becoming the accuracy bottleneck. The requirement that the fictitious-domain stabilization parameter be mesh-dependent to preserve high-order convergence is not specific to the present formulation; it is a general constraint across immersed and fictitious-domain methods. The $\phi$-FEM framework~\cite{duprez2020phi,duprez2023phifem}, though enforcing Dirichlet conditions strongly through trial-function multiplication, requires a boundary stabilization whose magnitude must be adapted to the mesh size to control the normal-flux mismatch across the immersed interface. Similarly, in CutFEM~\cite{burman2010,burman2012}, the ghost-penalty form carries an intrinsic $h$-dependent scaling as well. In the finite cell method~\cite{parvizian2007fcm,duster2008fcm,schillinger2015fcm}, the void stiffness parameter $\alpha$ is chosen sufficiently small (between $10^{-5}$ and $10^{-10}$) for high-order basis functions. In practice, setting $\gamma$ to a small constant ($10^{-10}$ in our numerical experiments, Section~\ref{sec:numerical}) satisfies this scaling for all mesh sizes of interest, since $\gamma^{1/2} \ll h^p$ for moderate $h$.

\subsection{Regularity of the body-fitted function}
\label{sec:exact:regularity}

Theorem~\ref{thm:error} assumes the BFF $\Phi$ to be sufficiently smooth on $\mathcal{D}$, so that the geometric coefficients $\Phi^2$ and $\Phi\,\nabla^2\Phi$ appearing in the weak form~\eqref{eq:weak_strong} are well-defined and the interpolation estimates used in the proof are valid. In practice, however, the signed distance function from which $\Phi$ is constructed is typically only $C^0$ continuous, as the SDF is continuous everywhere but its gradient is discontinuous across the medial axis (the skeleton of $\Omega$), where the nearest point on $\partial\Omega$ jumps between distinct boundary segments \cite{sethian1999level}.

The exponential map defined in Eq.~\eqref{eq:Phi_def} serves as a regularization of the original SDF as well. To see this, compute the gradient of $\Phi_\beta$:
\begin{equation}
    \nabla \Phi_\beta \;=\; -\frac{1}{\beta}\,\mathrm{e}^{\mathrm{SDF}/\beta}\,\nabla\mathrm{SDF}
    \;=\; \frac{1}{\beta}\,(1 - \Phi_\beta)\,\nabla\mathrm{SDF},
    \label{eq:grad_Phi_smooth}
\end{equation}
where we have used $|\nabla\mathrm{SDF}| = 1$ almost everywhere. At any point where $|\mathrm{SDF}| \gg \beta$, the factor $1 - \Phi_\beta = \mathrm{e}^{\mathrm{SDF}/\beta}$ is exponentially small, so $\nabla\Phi_\beta \approx 0$ regardless of the direction of $\nabla\mathrm{SDF}$. Consequently, if the parameter $\beta$ is chosen small enough that the medial-axis locus lies deep in the interior where $|\mathrm{SDF}| \gg \beta$, then $\Phi_\beta \approx 1$ and $\nabla\Phi_\beta \approx 0$ at those points, and the $C^0$ irregularity of the SDF is \emph{exponentially attenuated} in $\Phi_\beta$ and its derivatives. A practical guideline is to set $\beta \ll d_{\mathrm{med}}$, where $d_{\mathrm{med}}$ is the minimum distance from $\partial\Omega$ to the medial axis, so that the kinks of the SDF are effectively invisible to the body-fitted formulation. This regularization property is a key practical advantage of the exponential construction over, e.g., a linear ramp $\Phi = \max(-\mathrm{SDF}, 0)$, for which $\nabla\Phi$ inherits the full derivative discontinuity of the SDF.

\section{Immersed C-HiDeNN-TD using exact Dirichlet immersed form}
\label{sec:itd:exact}

For large-scale problems, the computational cost can be reduced by model order reduction. The C-HiDeNN-Tensor Decomposition (TD) method, that leverages a special form of tensor decomposition called canonical polyadic decomposition, can be applied directly to the Cartesian background $\mathcal{D}$.

The discrete trial and test functions follow the same TD structure as in the C-HiDeNN-TD framework \cite{lu2023chidenn,guo2024taps}:
\begin{equation}
    v^h(x,y,z) \;=\; \sum_{m=1}^{M} v_x^{(m)}(x)\,v_y^{(m)}(y)\,v_z^{(m)}(z),
    \qquad
    w^h \;=\; \delta v^h,
    \label{eq:cp_v}
\end{equation}
with the boundary lift $g^h$ represented as a rank-$M_g$ TD approximation
\begin{equation}
    g^h(x,y,z) \;=\; \sum_{m=1}^{M_g} g_x^{(m)}(x)\,g_y^{(m)}(y)\,g_z^{(m)}(z),
    \label{eq:cp_g}
\end{equation}
where each factor $v_\xi^{(m)}, g_\xi^{(m)}$ ($\xi \in \{x,y,z\}$) is represented using 1D C-HiDeNN shape functions.

To improve the efficiency of multi-dimensional integrals in the weak form, the key coefficients $\Phi$, $\Phi^2$, and $\Phi\,\nabla^2\Phi$ are also approximated by Tucker forms:
\begin{eqnarray}
	&& \Phi \approx \hat{\Phi}^{(1)} = \sum_{p=1}^{P_1} \sum_{q=1}^{Q_1} \sum_{k=1}^{K_1} \mathcal{H}^{(1)}_{pqk} \, \phi_{x}^{(p)}(x) \, \phi_{y}^{(q)}(y) \, \phi_{z}^{(k)}(z), \\
	&& \Phi^2 \approx \hat{\Phi}^{(2)} = \sum_{p=1}^{P_2} \sum_{q=1}^{Q_2} \sum_{k=1}^{K_2} \mathcal{H}^{(2)}_{pqk} \, \varphi_{x}^{(p)}(x) \, \varphi_{y}^{(q)}(y) \, \varphi_{z}^{(k)}(z), \\
	&& \Phi\,\nabla^2\Phi \approx \hat{\Phi}^{(3)} = \sum_{p=1}^{P_3} \sum_{q=1}^{Q_3} \sum_{k=1}^{K_3} \mathcal{H}^{(3)}_{pqk} \, \theta_{x}^{(p)}(x) \, \theta_{y}^{(q)}(y) \, \theta_{z}^{(k)}(z).
\label{tucker_phi}
\end{eqnarray}
It is important to note that $\hat{\Phi}^{(1)}$, $\hat{\Phi}^{(2)}$ and $\hat{\Phi}^{(3)}$ are decomposed independently using separate Tucker factorizations with distinct factor matrices and core tensors. Consequently, the algebraic relations $\hat{\Phi}^{(2)} = (\hat{\Phi}^{(1)})^2$ and $\hat{\Phi}^{(3)} = \hat{\Phi}^{(1)}\,\nabla^2\hat{\Phi}^{(1)}$ that hold for the exact geometric fields are generally not preserved in the Tucker-compressed representation. This deviation introduces an additional geometric consistency error. Enforcing these transformation relations by deriving $\hat{\Phi}^{(2)}$ and $\hat{\Phi}^{(3)}$ directly from $\hat{\Phi}^{(1)}$ would eliminate this source of error but may increase the tensor rank required to represent the derived fields.

Substituting \eqref{eq:cp_v} into \eqref{eq:weak_strong} and contracting against the Tucker cores \eqref{tucker_phi} produces three contributions to the assembled 1D operator $\mathbf{K}^{[\xi]}$ in each PGD direction $\xi \in \{x,y,z\}$: a stiffness term from $\Phi^2 + \gamma$, a mass-like term from $\Phi\,\nabla^2 \Phi$, and a coupling term contributing to the right-hand side from $\Phi$ and $g^h$.

To circumvent the curse of dimensionality, the trial space is enriched mode by mode using PGD. At enrichment step $M$, the solution is split into a known part (the first $M-1$ modes) and an unknown rank-1 increment,
\begin{equation}
    v^h(x,y,z) \;=\;
    \underbrace{\sum_{m=1}^{M-1} v_x^{(m)}(x)\,v_y^{(m)}(y)\,v_z^{(m)}(z)}_{v^{<}(\mathbf{x})}
    \;+\; \bar{v}_x(x)\,\bar{v}_y(y)\,\bar{v}_z(z),
    \label{eq:pgd_split_exact}
\end{equation}
and the three 1D factors $\bar{v}_x, \bar{v}_y, \bar{v}_z$ are computed by an alternating fixed-point iteration. Restricting the variation to the $x$ component, $\delta v^h = \delta\bar{v}_x(x)\,\bar{v}_y(y)\,\bar{v}_z(z)$, and substituting into \eqref{eq:weak_strong} yields a 1D linear system
\begin{equation}
    \mathbf{K}^{[x]} \mathbf{v}_x \;=\; \mathbf{q}^{[x]},
    \label{eq:pgd_1d_system_exact}
\end{equation}
in which the operator $\mathbf{K}^{[x]}$ and right-hand side $\mathbf{q}^{[x]}$ are assembled by contracting the Tucker cores \eqref{tucker_phi} against scalars obtained from 1D quadrature in the fixed directions $y$ and $z$. The procedure is repeated cyclically over $x \to y \to z$ until the relative change in the rank-1 mode falls below a prescribed tolerance; the new mode is then appended to $v^{<}$ and the next enrichment step is initiated.


\subsection{Error estimate for immersed TD}

In the immersed TD formulation, the BFF $\Phi$ and its derived fields $\Phi^2$ and $\Phi\,\nabla^2\Phi$ are approximated respectively by independent Tucker decompositions $\hat{\Phi}^{(1)}$, $\hat{\Phi}^{(2)}$ and $\hat{\Phi}^{(3)}$ as defined in Eq.~\eqref{tucker_phi}. This geometric compression introduces additional deviations whose impact is quantified by the following error estimate. Define the Tucker perturbation fields
\begin{equation} \label{eq:delta_def}
    \delta_1 = \hat{\Phi}^{(1)} - \Phi,\qquad
    \delta_2 = \hat{\Phi}^{(2)} - \Phi^2,\qquad
    \delta_3 = \hat{\Phi}^{(3)} - \Phi\,\nabla^2\Phi.
\end{equation}

\begin{theorem}\label{thm:error_tucker}
Let $\delta_1, \delta_2, \delta_3$ be the Tucker perturbation fields defined in~\eqref{eq:delta_def}. Under the same assumptions as Theorem~\ref{thm:error}, the discrete C-HiDeNN solution $u^h$ obtained from Eq.~\eqref{eq:weak_strong} with the Tucker-compressed coefficients satisfies
\begin{equation} \label{eq:error_estimate_tucker}
\begin{aligned}
    \| u - u^h \|_{\gamma} \;\leq\; &\; C_1 h^{p}\, \|\tilde{v}\|_{H^{p+1}(\mathcal{D})} \\
    &+ C_2 \Bigl(\sup_{\mathbf{x}\in\Omega} \gamma(\mathbf{x}) + \|\delta_2\|_{L^\infty(\Omega)}\Bigr)
       \sup_{\eta^h \in \mathcal{V}^h\setminus\{0\}} \dfrac{\|\nabla \eta^h/\Phi\|_{L^2(\Omega)}}{\|\nabla \eta^h\|_{L^2(\Omega)}} \cdot \|u\|_{H^1(\Omega)} \\
    &+ C_3 \sqrt{\sup_{\mathbf{x}\in\mathcal{D}\setminus\Omega} \gamma(\mathbf{x}) + \|\delta_2\|_{L^\infty(\mathcal{D}\setminus\Omega)} + \|\delta_3\|_{L^\infty(\mathcal{D}\setminus\Omega)}}\; \|\hat{v} - \tilde{v}\|_{H^1(\mathcal{D}\setminus\overline{\Omega})} \\
    &+ C_4 \|\delta_3\|_{L^\infty(\mathcal{D})} \sup_{\eta^h \in \mathcal{V}^h\setminus\{0\}} \dfrac{\|\eta^h/\Phi\|_{L^2(\mathcal{D})}}{\|\eta^h\|_\gamma} \cdot \|\hat{v}\|_{L^2(\mathcal{D})} \\
    &+ C_5 \|\delta_1\|_{L^\infty(\mathcal{D})} \sup_{\eta^h \in \mathcal{V}^h\setminus\{0\}} \dfrac{\|\eta^h/\Phi\|_{L^2(\mathcal{D})}}{\|\eta^h\|_\gamma} \cdot \|\nabla^2 g^h + b\|_{L^2(\mathcal{D})},
\end{aligned}
\end{equation}
where $C_1, C_2, C_3, C_4, C_5 > 0$ are constants independent of $h$ and of the Tucker rank.
\end{theorem}

Comparing Theorem~\ref{thm:error_tucker} with Theorem~\ref{thm:error}, the perturbation $\delta_2$ enters the estimate in exactly the same manner as the void stabilization parameter $\gamma$ (second and third terms). This reveals a fundamental structural analogy: the Tucker compression error in $\hat{\Phi}^{(2)}$ acts as an \emph{effective additional stabilization} whose $L^\infty$-norm combines additively with $\gamma$. Consequently, the accuracy of the Tucker-compressed formulation is controlled by the \emph{total} stabilization magnitude $\gamma + \delta_2$, and the Tucker rank must be chosen such that $\|\delta_2\|_{L^\infty}$ remains commensurate with $\gamma$ to avoid degrading the convergence rate. The $\delta_3$ perturbation, arising from the independent compression of $\Phi\,\nabla^2\Phi$, contributes a mass-like error term (fourth term) that scales with $\|\delta_3\|_{L^\infty}$. The $\delta_1$ perturbation propagates directly through the right-hand side (fifth term). A complete proof is given in Appendix~\ref{appenx:Proof_tucker}.

\begin{remark}[Generality of the geometric error terms]
The perturbation fields $\delta_1, \delta_2, \delta_3$ in Theorem~\ref{thm:error_tucker} are not restricted to Tucker decomposition errors. They accommodate any source of inaccuracy in the BFF and its derived geometric coefficients, making the estimate applicable to a broad class of geometric approximation strategies:
\begin{itemize}
    \item \textbf{Voxel-based SDF discretization.} When the signed distance function is computed on a finite-resolution voxel grid, the resulting $\Phi$ constructed via Eq.~\eqref{eq:Phi_def} inherits a discretization error proportional to the voxel spacing. This error propagates into the geometric coefficients and is captured by the $\delta_1$ (fifth) and $\delta_2$ (second/third) terms.
    \item \textbf{Interpolation-based approximations.} When $\Phi$ is approximated by interpolation from a set of sample points (e.g., using radial basis functions, neural-network implicit representations, or spline fitting), the interpolation error enters the estimate through the same mechanisms.
    \item \textbf{SDF regularization or smoothing.} Any post-processing of the SDF prior to constructing $\Phi$ (e.g., Gaussian filtering to suppress noise in image-derived geometries) introduces perturbations to $\Phi$, $\Phi^2$, and $\Phi\,\nabla^2\Phi$ whose effects are likewise bounded by the three perturbation channels.
\end{itemize}
In each case, provided the approximation errors $\|\delta_1\|_{L^\infty(\mathcal{D})}$, $\|\delta_2\|_{L^\infty(\mathcal{D})}$, and $\|\delta_3\|_{L^\infty(\mathcal{D})}$ can be estimated, either analytically or numerically, Theorem~\ref{thm:error_tucker} yields a quantitative bound on the resulting solution error.
\end{remark}

\section{Numerical examples}
\label{sec:numerical}

In immersed methods, the diffuse interface (the narrow band where $0 < \Phi < 1$) requires higher-order quadrature to resolve the sharp gradients of the BFF and its derivatives. Applying uniformly dense Gauss quadrature over all elements is unnecessarily expensive: elements deep inside $\Omega$ (where $\Phi \approx 1$ and its derivatives are smooth) and elements in the void (where $\Phi = 0$) do not benefit from the extra integration points.

We adopt a two-level adaptive Gauss quadrature strategy. A pre-classification step evaluates the nodal values of $\Phi$ on each element and labels elements intersecting the transition zone (those with at least one node satisfying $\varepsilon_{\text{ext}} < \Phi < \varepsilon_{\text{int}}$) as boundary elements; all others are non-boundary elements. The implementation follows Algorithm~\ref{alg:classify}.

\begin{algorithm}[htbp]
\caption{Element classification for adaptive Gauss quadrature}
\label{alg:classify}
\begin{algorithmic}[1]
\State \textbf{Input:} nodal coordinates, mesh topology, $\Phi(\mathbf{x})$, thresholds $\varepsilon_{\text{ext}}$, $\varepsilon_{\text{int}}$
\State Evaluate $\Phi$ at all mesh nodes: $\Phi_i \leftarrow \Phi(\mathbf{x}_i)$
\For{each element $e$ with nodes $\mathcal{N}_e$}
    \State $\Phi_{\min}^{(e)} \leftarrow \min_{i \in \mathcal{N}_e} \Phi_i$, $\quad \Phi_{\max}^{(e)} \leftarrow \max_{i \in \mathcal{N}_e} \Phi_i$
    \If{$\Phi_{\max}^{(e)} > \varepsilon_{\text{ext}}$ \textbf{and} $\Phi_{\min}^{(e)} < \varepsilon_{\text{int}}$}
        \State Mark $e$ as \textbf{boundary}
    \Else
        \State Mark $e$ as \textbf{non-boundary}
    \EndIf
\EndFor
\State \textbf{Output:} boundary element set $\mathcal{E}_B$, non-boundary element set $\mathcal{E}_N$
\end{algorithmic}
\end{algorithm}

In the assembly of the stiffness matrix and force vector, elements in $\mathcal{E}_B$ are integrated with $n_g^{\text{bnd}}$ Gauss points per direction ($20 \times 20$ in the present computations), while elements in $\mathcal{E}_N$ use a coarser $n_g^{\text{int}}$ rule ($6 \times 6$). The error computation employs a slightly enriched rule ($n_g^{\text{bnd}}+2$, $n_g^{\text{int}}+2$) to avoid quadrature-dominated error estimates. For the multiply-connected void domain, the boundary element set $\mathcal{E}_B$ typically accounts for approximately $10$--$15\%$ of the total elements; the adaptive strategy therefore reduces the quadrature cost by roughly a factor of $5$--$7$ compared with a uniform $20 \times 20$ rule, with no measurable loss of accuracy. Table~\ref{tab:gauss_adaptive} summarizes the quadrature configuration.

\begin{table}[htbp]
\centering
\caption{Adaptive Gauss quadrature configuration}
\label{tab:gauss_adaptive}
\begin{tabular}{lcc}
\hline
\textbf{Element type} & \textbf{Criterion} & \textbf{Gauss points}  \\
\hline
Boundary  & $\Phi_{\max} > \varepsilon_{\text{ext}} \;\text{and}\; \Phi_{\min} < \varepsilon_{\text{int}}$ & $20 \times 20$  \\
Non-boundary & Otherwise & $6 \times 6$ \\
\hline
\end{tabular}
\end{table}

\subsection{2D circular / annular computational domain}

The exact Dirichlet immersed formulation is evaluated on two 2D Poisson benchmark problems using the BFF $\Phi_\beta$ defined in Eq.~\eqref{eq:Phi_def}, embedded in the square background domain $\mathcal{D} = [-1,1] \times [-1,1]$. The void stabilization is set to $\gamma(\mathbf{x})=10^{-10}$ in $\mathcal{D}$, which is sufficiently small to be considered negligible within the range of mesh resolutions of interest and does not affect the observed convergence rates. The C-HiDeNN radial basis expansion parameter is $r_a=50$ throughout.

\subsubsection{2D circular computational domain}
\label{sec:numerical:circle}

The first benchmark is the Poisson equation on a circular domain $\Omega = \{ (x,y) \mid \sqrt{x^2+y^2} < R\}$ of radius $R=0.7$, subject to homogeneous Dirichlet boundary conditions. The signed distance function for this geometry admits the closed form
\begin{equation}
    \mathrm{SDF}(x,y) = \sqrt{x^2 + y^2} - R,
    \label{eq:sdf_circle}
\end{equation}
with the sign convention illustrated in Fig.~\ref{fig:sdf} (negative inside $\Omega$, positive outside). The BFF $\Phi_\beta$ is then constructed from Eq.~\eqref{eq:Phi_def}. Figure~\ref{fig:circle_geom} shows the SDF and the resulting BFF $\Phi_\beta$ for the circular domain.

\begin{figure}[htbp]
\centering
\includegraphics[width=0.48\textwidth]{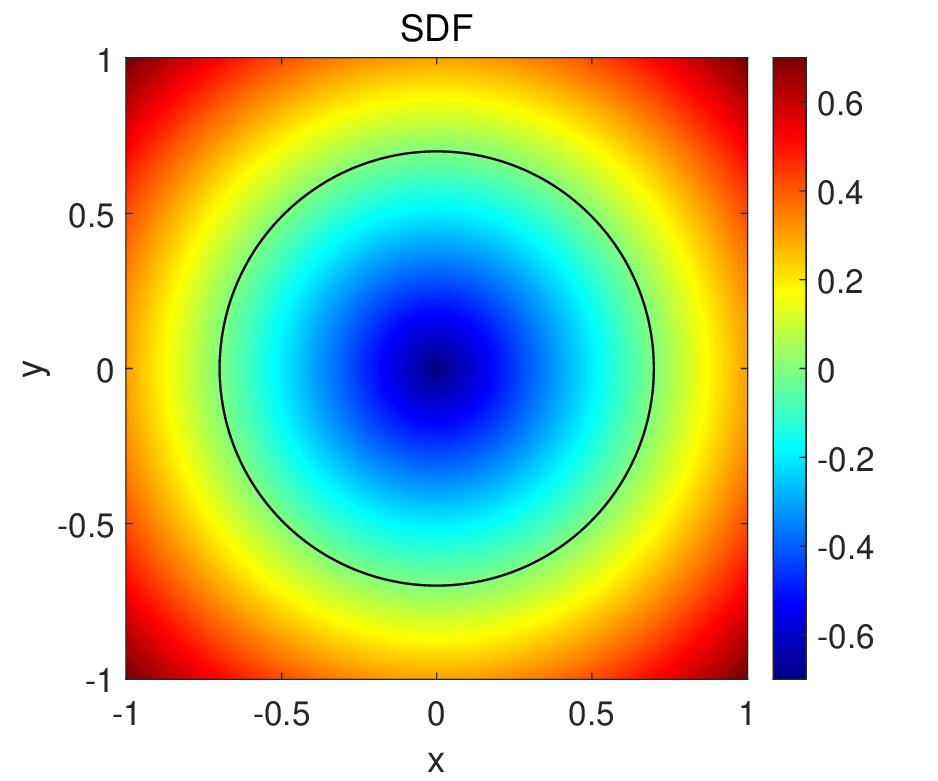}
\hfill
\includegraphics[width=0.48\textwidth]{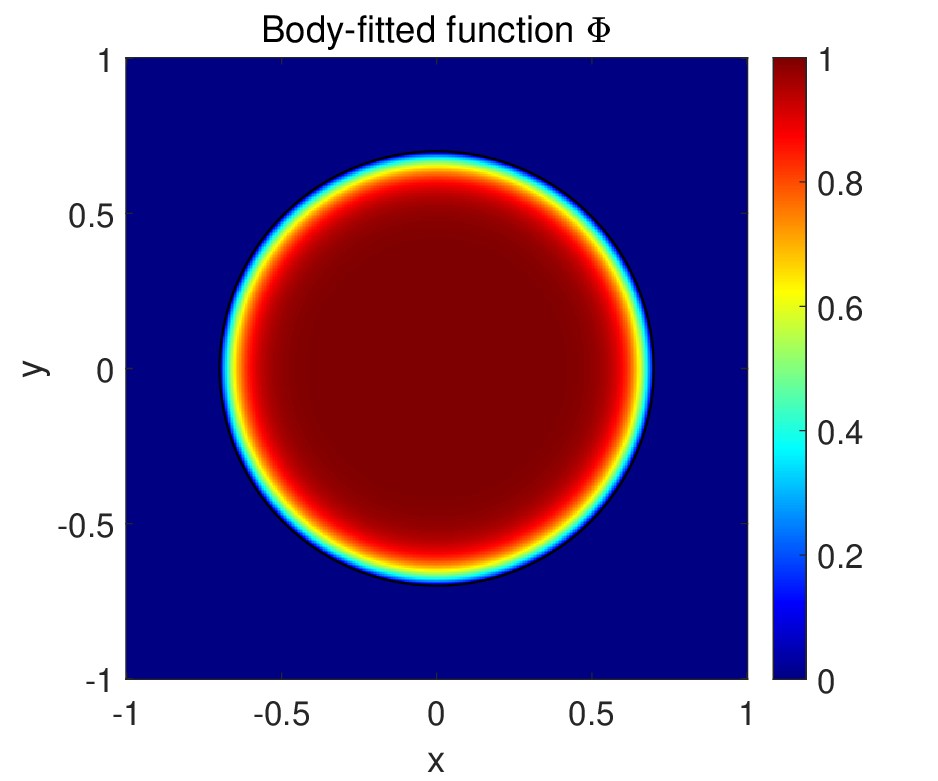}
\caption{SDF and BFF $\Phi_\beta$ for the 2D circular domain ($R=0.7$). The black solid line marks the domain boundary $\partial\Omega$.}
\label{fig:circle_geom}
\end{figure}

The exact solution is chosen as
\begin{equation}
    u^{\text{Ext}}(x,y) = R^2 - x^2 - y^2,
    \label{eq:circle_exact}
\end{equation}
which vanishes on $\partial\Omega$ and yields a constant source term $b(x,y)=4$. 

Figure~\ref{fig:circle_conv} shows the convergence of the relative $L^2$ and energy errors for $p = 1, 2, 3$ with $\beta = 0.01$, under uniform mesh refinement ($n_x = 8, 16, 32, 64, 128, 256$). The estimated convergence rates are summarized in Table~\ref{tab:circle_conv_rates}.

\begin{table}[htbp]
\centering
\caption{Estimated convergence rates for the circular domain ($\beta = 0.01$)}
\label{tab:circle_conv_rates}
\begin{tabular}{lccc}
\hline
 & \textbf{$L^2$ rate (last 4 points)} & \textbf{Energy rate (last 4 points)} & \textbf{Theoretical optimum} \\
\hline
$p = 1$ & $2.2$ & $1.2$ & $2.0$ / $1.0$ \\
$p = 2$ & $2.9$ & $1.8$ & $3.0$ / $2.0$ \\$p = 3$ & $3.4$ & $2.5$ & $4.0$ / $3.0$ \\
\hline
\end{tabular}
\end{table}

For $p = 1$ and $p = 2$, the estimated rates are close to their respective theoretical optima. For $p = 3$, the rates are $\mathcal{O}(h^{3.4})$ in $L^2$-norm and $\mathcal{O}(h^{2.5})$ in energy-norm, below the nominal $\mathcal{O}(h^{4})$ and $\mathcal{O}(h^{3})$. This suboptimality is attributed to the combined effects of the reduced regularity of $\Phi_\beta$ in the transition layer, the finite accuracy of the adaptive boundary-integration quadrature, and the fact that, on the finer meshes, the error is sufficiently small that these previously subdominant error sources become visible.

\begin{figure}[htbp]
\centering
\includegraphics[width=0.48\textwidth]{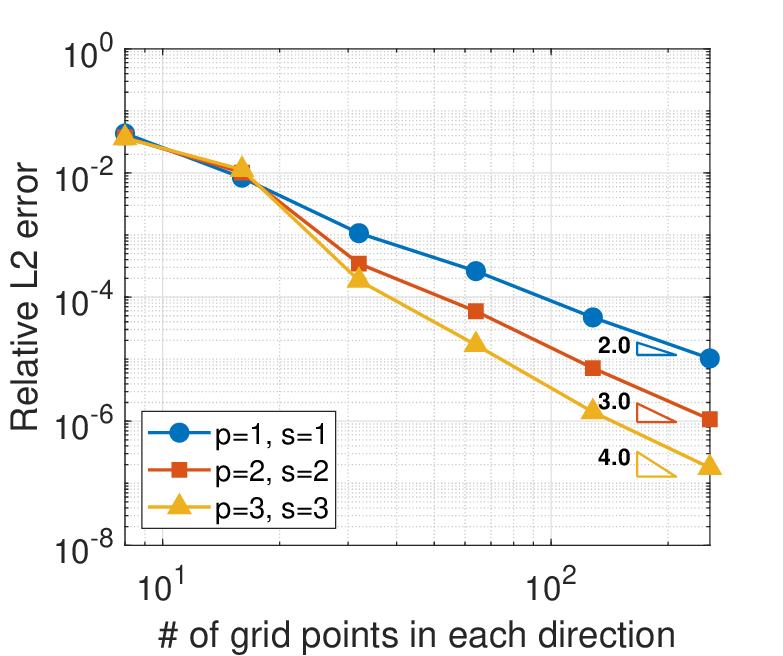}
\hfill
\includegraphics[width=0.48\textwidth]{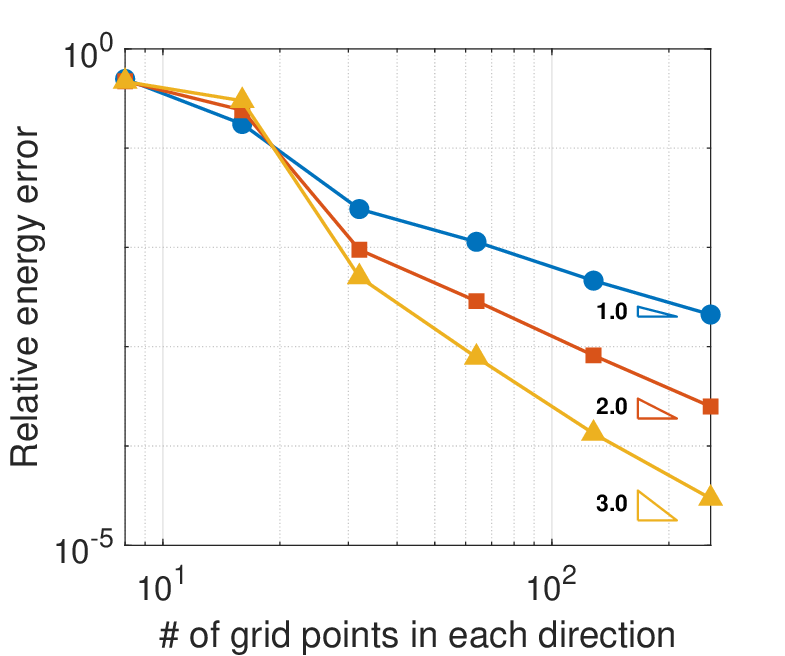}
\caption{Convergence of the relative $L^2$ error (left) and relative energy error (right) for the 2D circular domain ($\beta = 0.01$). Triangles indicate the theoretical optimal convergence rates.}
\label{fig:circle_conv}
\end{figure}

We isolate the role of the void stabilization parameter $\gamma$ by fixing the mesh ($128 \times 128$) and sweeping $\gamma$ over nine decades, from $10^{-2}$ to $10^{-12}$. All other parameters are held constant ($p=3$, $s=3$, adaptive Gauss quadrature with $20\times20$ boundary points and $6\times6$ interior points).

\begin{figure}[htbp]
\centering
\includegraphics[width=0.48\textwidth]{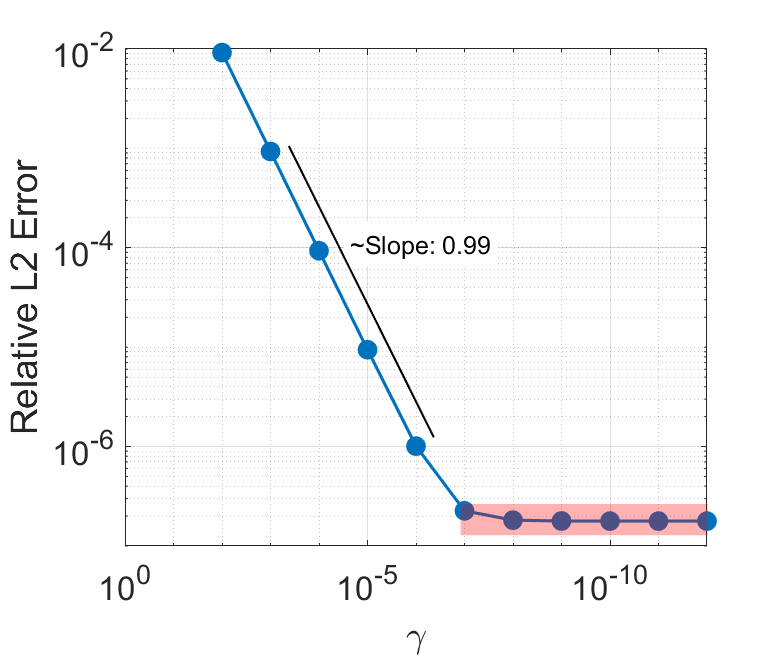}
\hfill
\includegraphics[width=0.48\textwidth]{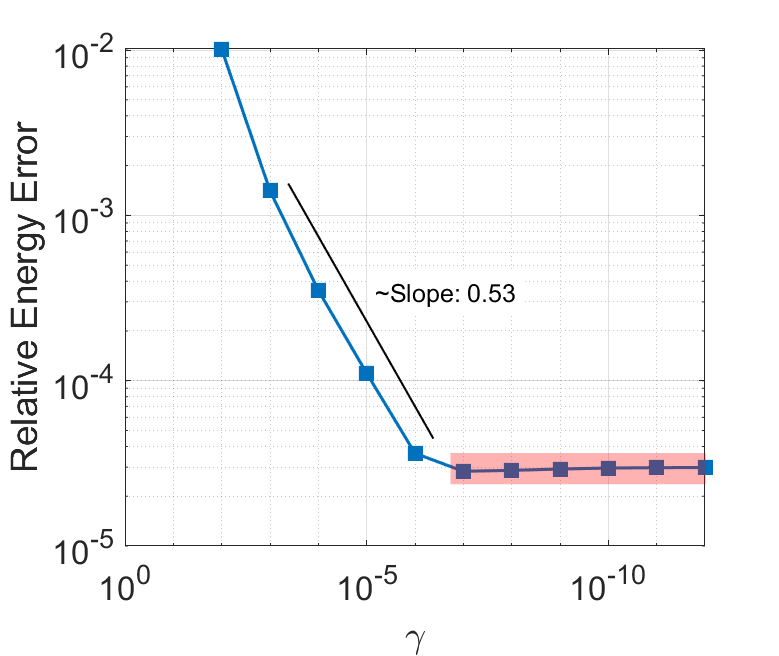}
\caption{Relative $L^2$ error (left) and relative energy error (right) versus void stabilization parameter $\gamma$ on a fixed $128\times 128$ mesh ($p=3$). The fitted slope over points 2-5 ($\gamma = 10^{-3}$ to $10^{-6}$) is annotated on each panel. The red-shaded region indicates the \textbf{Discretization-dominated regime ($\gamma \lesssim 10^{-7}$).}, in which the total error saturates.}
\label{fig:gamma_study}
\end{figure}

Figure~\ref{fig:gamma_study} plots the relative $L^2$ and energy errors as functions of $\gamma$ on a log-log scale. Two regimes are evident:
(1) \textbf{Stabilization-dominated regime} ($\gamma \gtrsim 10^{-7}$). 
For large $\gamma$, the void stabilization distorts the solution throughout the physical domain $\Omega$. The energy error grows as $\mathcal{O}(\gamma^{0.5})$, consistent with the $\sqrt{\gamma}$ scaling of the third term in Corollary 2, where a linear regression over points 2-5 (corresponding to $\gamma = 10^{-3}$ through $10^{-6}$) yields a fitted slope of approximately $0.5$ for the energy norm. The $L^2$ error exhibits a fitted slope of approximately $1$ over the same interval.
(2) \textbf{Discretization-dominated regime ($\gamma \lesssim 10^{-7}$).} As $\gamma$ is further reduced, both the $L^2$ and energy errors saturate to a plateau. In this regime the $\gamma$-induced error has fallen below the intrinsic discretization error of the $128\times128$ mesh, so further reduction of $\gamma$ yields no improvement in accuracy. This shows that choosing $\gamma = 10^{-10}$ is appropriate in the above convergence study.


\subsubsection{2D annular domain with an eccentric hole}
\label{sec:numerical:void}

The second benchmark introduces a non-convex, multiply-connected geometry: a circular outer boundary of radius $R_1=0.7$ containing an eccentric circular hole of radius $R_2=0.2$ centered at $(x_0, y_0)=(0.15, 0)$. For a multiply-connected domain defined by the set difference of an outer and an inner region, the signed distance function combines the individual distance fields via the maximum operator:
\begin{equation}
    \mathrm{SDF}(x,y) = \max\!\Bigl(\sqrt{x^2 + y^2} - R_1,\; R_2 - \sqrt{(x - x_0)^2 + y^2}\Bigr),
    \label{eq:sdf_void}
\end{equation}
which is negative when both $\sqrt{x^2+y^2} < R_1$ (inside the outer circle) and $\sqrt{(x-x_0)^2+y^2} > R_2$ (outside the inner hole) are satisfied. See Fig.~\ref{fig:sdf} for the sign convention. The BFF follows from Eq.~\eqref{eq:Phi_def}. Figure~\ref{fig:void_geom} shows the SDF and the resulting BFF $\Phi_\beta$ for the annular domain.

\begin{figure}[htbp]
\centering
\includegraphics[width=0.48\textwidth]{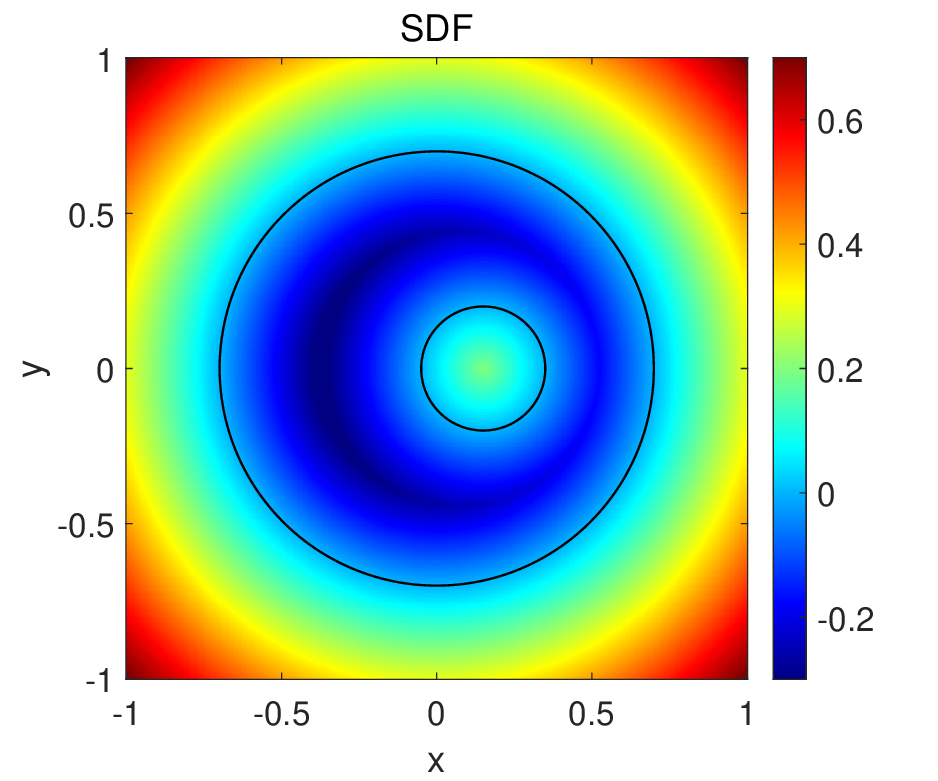}
\hfill
\includegraphics[width=0.48\textwidth]{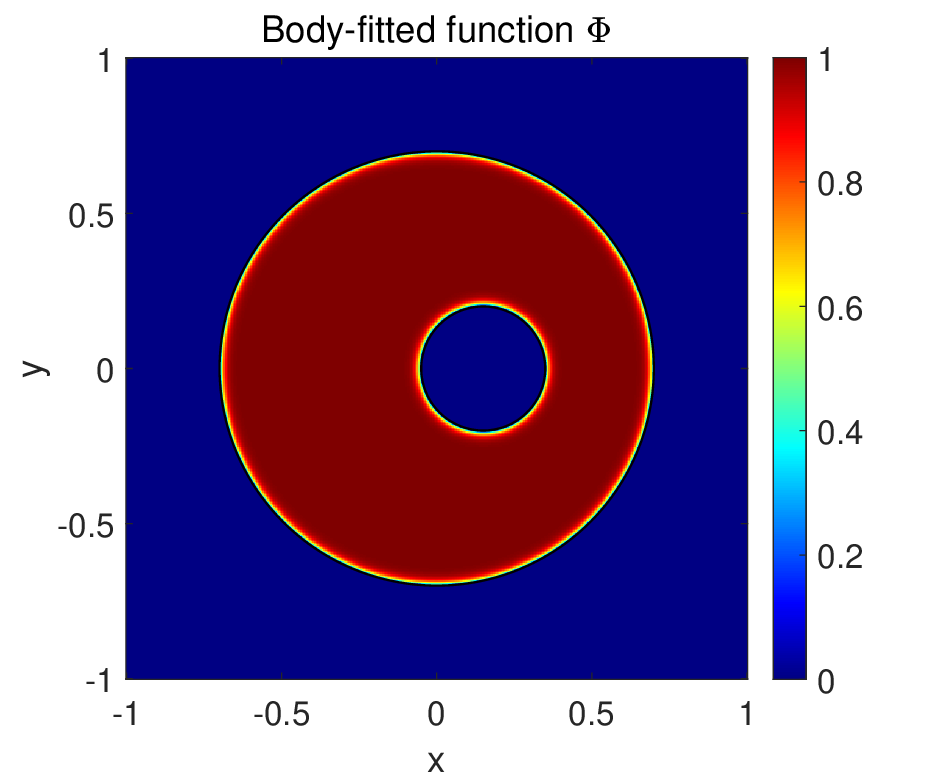}
\caption{SDF and BFF $\Phi_\beta$ for the 2D annular domain with eccentric hole ($R_1=0.7$, $R_2=0.2$, $x_0=0.15$, $\beta=0.01$). The black solid lines mark the outer boundary and the inner hole.}
\label{fig:void_geom}
\end{figure}

The exact solution,
\begin{equation}
    u^{\text{Ext}}(x,y) = -(R_1^2 - x^2 - y^2)\,\bigl(R_2^2 - (x - x_0)^2 - y^2\bigr),
    \label{eq:void_exact}
\end{equation}
is constructed to vanish on both the outer boundary and the inner hole, satisfying the homogeneous Dirichlet condition. The corresponding source term $b(x,y)=-\Delta u^{\text{Ext}}$ is obtained analytically. 



The key challenge in this example is the accurate resolution of the narrow ligament between the outer boundary and the eccentric hole, which tests the geometric fidelity of the body-fitted description under aggressive mesh refinement. Figure~\ref{fig:void_conv} shows the convergence of the relative $L^2$ and energy errors for polynomial orders $p = 1, 2, 3$, all using the adaptive Gauss quadrature of Table~\ref{tab:gauss_adaptive}. The mesh size is defined as $h = 2/n_x$, where $n_x$ is the number of elements in the $x$-direction.

For $p = 1$, the $L^2$ error converges approximately as $\mathcal{O}(h^{2.2})$ and the energy error as $\mathcal{O}(h^{1.2})$, matching the expected optimal rates $\mathcal{O}(h^{p+1})$ and $\mathcal{O}(h^{p})$ respectively. For $p = 2$, the $L^2$ error decays as $\mathcal{O}(h^{2.8})$ and the energy error as $\mathcal{O}(h^{1.7})$, again in agreement with the theory. For $p = 3$, the $L^2$ rate is approximately $\mathcal{O}(h^{3.4})$ and the energy rate is $\mathcal{O}(h^{2.5})$; the slight suboptimality relative to the nominal $\mathcal{O}(h^4)$ and $\mathcal{O}(h^3)$ rates is attributed to the narrow-ligament geometry, where the element size must be sufficiently small to resolve the curvature of both the outer and inner boundaries simultaneously; this condition is met only on the finest meshes. The convergence rates are summarized in Table~\ref{tab:void_conv_rates}.

\begin{table}[htbp]
\centering
\caption{Estimated convergence rates for the annular domain (last three mesh refinements)}
\label{tab:void_conv_rates}
\begin{tabular}{lccc}
\hline
 & \textbf{$L^2$ rate (last 4 points)} & \textbf{Energy rate (last 4 points)} & \textbf{Theoretical optimum} \\
\hline
$p = 1$ & $2.2$ & $1.2$ & $2.0$ / $1.0$ \\
$p = 2$ & $2.8$ & $1.7$ & $3.0$ / $2.0$ \\
$p = 3$ & $3.4$ & $2.5$ & $4.0$ / $3.0$ \\
\hline
\end{tabular}
\end{table}

\begin{figure}[htbp]
\centering
\includegraphics[width=0.48\textwidth]{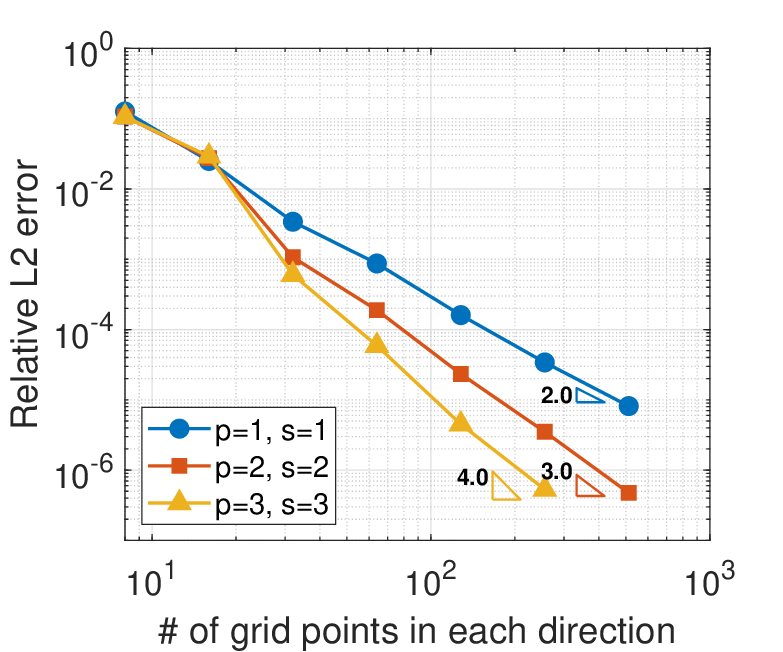}
\hfill
\includegraphics[width=0.48\textwidth]{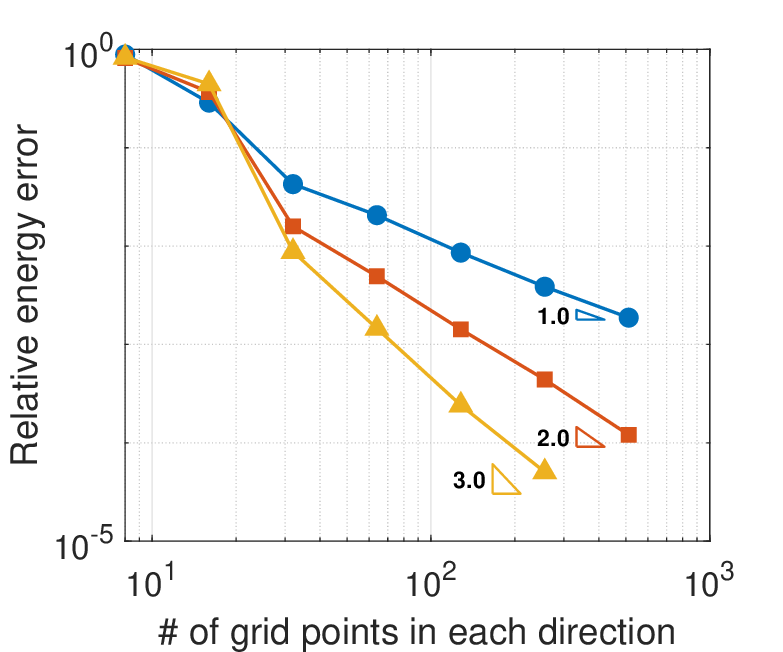}
\caption{Convergence of the relative $L^2$ error (left) and relative energy error (right) for the 2D annular domain with eccentric hole ($R_1=0.7$, $R_2=0.2$, $x_0=0.15$, $\beta=0.01$). The dashed lines indicate the theoretical optimal convergence rates $\mathcal{O}(h^{p+1})$ for $L^2$ and $\mathcal{O}(h^{p})$ for the energy norm.}
\label{fig:void_conv}
\end{figure}

\subsection{2D complex geometries}
\label{sec:numerical:starfish}

\subsubsection{2D starfish computational domain with a central hole}
\label{sec:numerical:starfish}

The next benchmark considers the Poisson equation on a more complex geometry: a 7-petaled starfish domain with a central hole. The domain is embedded within a square background box $\mathcal{D} = [0,2] \times [0,2]$ centered at $\mathbf{c}=(1,1)$. Using local polar coordinates $\rho = \sqrt{(x-1)^2+(y-1)^2}$ and $\theta = \operatorname{atan2}(y-1, x-1)$, the computational domain is defined as
\begin{equation}
    \Omega = \Big\{ (x,y) \mid r_{\text{hole}} \le \rho \le R(\theta) \Big\},
\end{equation}
where the outer boundary is given by $R(\theta) = r_{\text{base}} + \alpha \cos(7\theta)$. The geometric parameters are set to a mean petal radius $r_{\text{base}} = 0.70$, petal amplitude $\alpha = 0.24$, and a central hole radius $r_{\text{hole}} = 0.16$. Consequently, the domain boundary $\partial\Omega$ consists of two distinct components: the outer wavy curve and the inner circular void.

Homogeneous Dirichlet boundary conditions ($u=0$) are prescribed on all of $\partial\Omega$, with constant body force $b = 1.0$ as in Eq.~\eqref{eq:poisson_strong}.

The Dirichlet conditions are enforced geometrically through the BFF $\Phi$, utilizing the ansatz $u^h = \Phi v^h$ (homogeneous case) with $v^h$ pinned to $0$ on the boundary of the background box. The BFF is defined via Eq.~\eqref{eq:Phi_def} with transition parameter $\beta = 0.002$. The immersed weak form follows Eq.~\eqref{eq:weak_strong}:
\begin{equation}
    A_\gamma(w^h, v^h) = \int_{\mathcal{D}} (\Phi^2 + \gamma) \nabla w^h \cdot \nabla v^h - \Phi \nabla^2\Phi \, w^h v^h \,\mathrm{d}\mathcal{D},
\end{equation}
with the corresponding load functional $L(w^h) = \int_{\mathcal{D}} \Phi w^h b \,\mathrm{d}\mathcal{D}$. The void stabilization parameter is set to $\gamma = 10^{-10}$. 

Unlike the circular benchmark, this problem does not possess a closed-form exact solution. Instead, the numerical results are compared against a high-fidelity reference solution obtained using a body-fitted $P_1$ finite element solve on an identical triangulated starfish domain via FEniCSx. Figure~\ref{fig:starfish_solution} visualizes this comparison, displaying the reference field, the immersed C-HiDeNN solution, and the point-wise absolute error.

\begin{figure}[htbp]
\centering
\includegraphics[width=\textwidth]{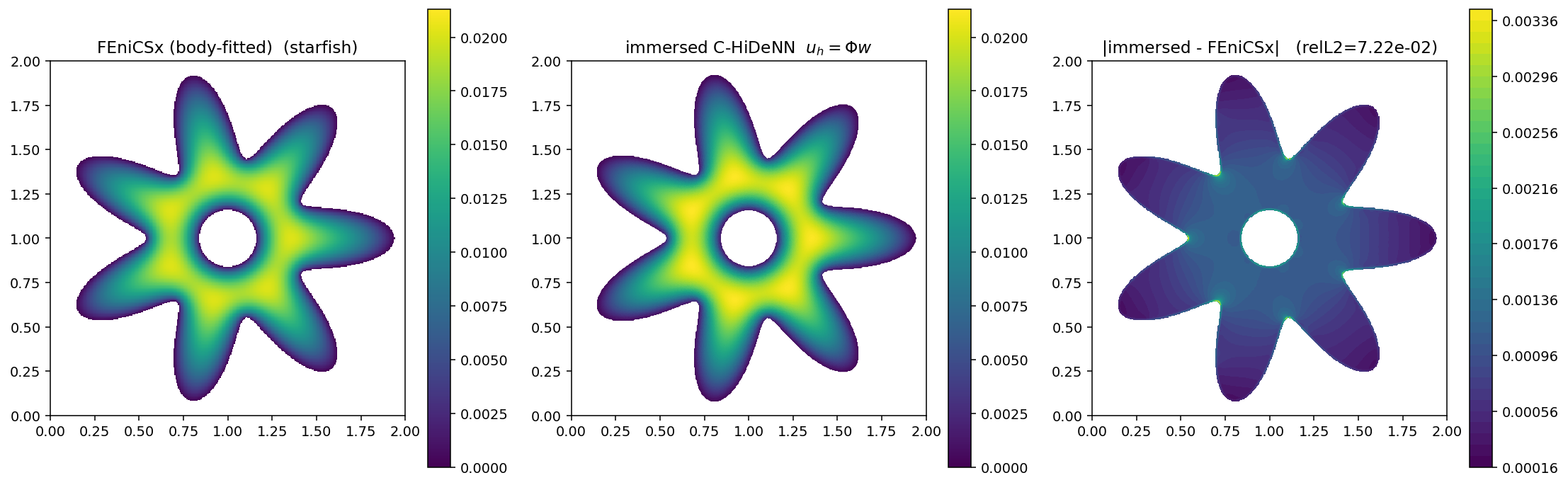}
\caption{Solution fields for the 2D starfish domain with a central hole. Left: Body-fitted FEniCSx reference solution. Middle: Immersed C-HiDeNN solution ($u^h = \Phi v^h$). Right: Point-wise absolute error between the immersed and reference solutions.}
\label{fig:starfish_solution}
\end{figure}

\subsubsection{2D gear computational domain with a central hole}
\label{sec:numerical:gear}

To further evaluate the framework's capability in resolving sharp geometric features, the third benchmark considers the Poisson equation on a 20-tooth gear domain with a central hole. The domain is embedded within the same square background box $\mathcal{D} = [0,2] \times [0,2]$ centered at $\mathbf{c}=(1,1)$. Using the local polar coordinates $\rho$ and $\theta$ defined previously, the computational domain is defined as
\begin{equation}
    \Omega = \Big\{ \text{inside the 20-tooth gear polygon} \Big\} \setminus \Big\{ \rho < r_{\text{hole}} \Big\}.
\end{equation}

The outer boundary is defined as a piecewise-linear polygon constructed from $80$ vertices located at angles $\theta_i = 2\pi i/80$. The corresponding radii cycle with a period of $4$ in the sequence $[r_{\text{in}}, r_{\text{out}}, r_{\text{out}}, r_{\text{in}}]$. This generates square-like teeth, where each tooth rises from the root radius to a flat tip at the tooth radius and back via straight segments. The geometric parameters are set to a tooth-tip radius $r_{\text{out}} = 0.90$, a tooth-root radius $r_{\text{in}} = 0.70$, and a central hole radius $r_{\text{hole}} = 0.50$. Thus, the domain boundary $\partial\Omega$ consists of the outer 20-tooth gear polygon and the inner circular void.

In direct contrast to the smooth analytic curve of the starfish domain, the gear's outer edge introduces multiple sharp, re-entrant corners at the $80$ polygon vertices. These discontinuities present a significantly more stringent geometric challenge for resolving the BFF $\Phi$.

The physical problem remains identical to the starfish benchmark (constant body force $b = 1.0$, homogeneous Dirichlet boundary conditions), with the BFF, immersed weak form, and stabilization parameters ($\gamma = 10^{-10}$, $\beta = 0.002$) retained from the previous examples. 


Similar to the starfish benchmark, the numerical results are compared against a high-fidelity reference solution obtained via a body-fitted $P_1$ finite element solve on an identical triangulated gear domain using FEniCSx. Figure~\ref{fig:gear_solution} displays the reference solution, the immersed C-HiDeNN solution, and the point-wise absolute error, illustrating the framework's performance around the sharp polygonal features.

\begin{figure}[htbp]
\centering
\includegraphics[width=\textwidth]{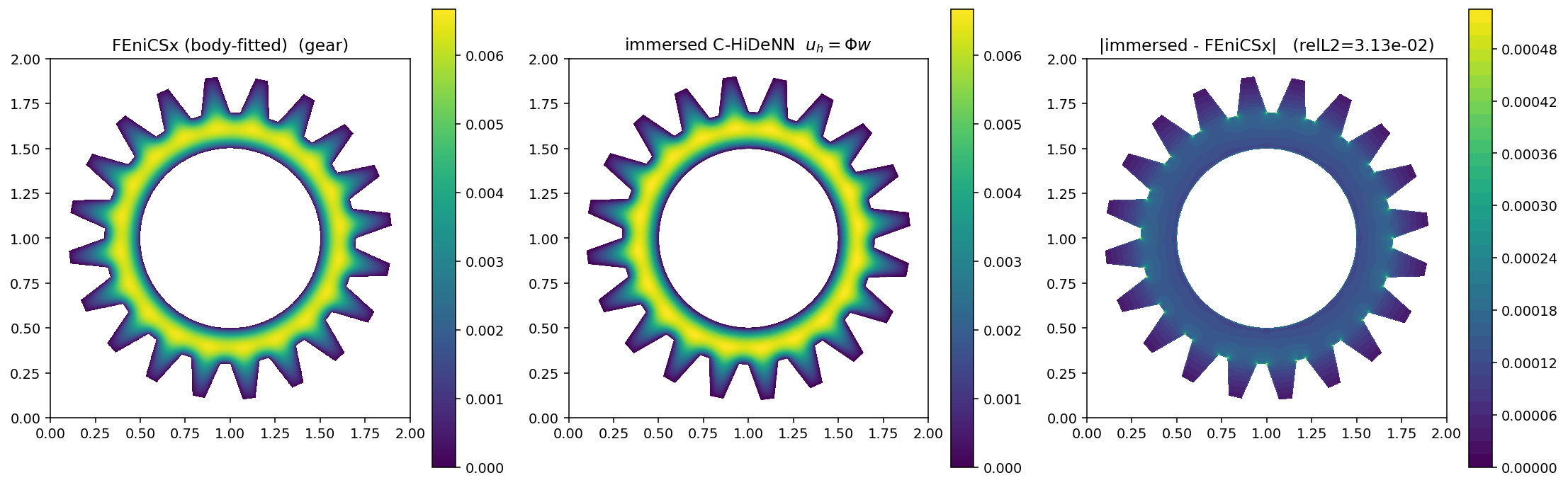}
\caption{Solution fields for the 2D gear domain with a central hole. Left: Body-fitted FEniCSx reference solution. Middle: Immersed C-HiDeNN solution ($u^h = \Phi v^h$). Right: Point-wise absolute error between the immersed and reference solutions.}
\label{fig:gear_solution}
\end{figure}

\subsubsection{3D sphere with a concentric spherical void}
\label{sec:numerical:void3d}

The final benchmark extends the immersed body-fitted formulation to three dimensions
on a multiply-connected solid: a spherical shell formed by an outer sphere of radius
$R_1=0.7$ with a concentric spherical void of radius $R_2=0.3$, both centered at
$\mathbf{x}_c=(1,1,1)$ inside the background box $[0,2]^3$. Writing
$r=\lVert\mathbf{x}-\mathbf{x}_c\rVert$, the signed distance function again combines the
outer and inner fields through the maximum operator,
\begin{equation}
    \mathrm{SDF}(\mathbf{x}) = \max\!\bigl(r - R_1,\; R_2 - r\bigr),
    \label{eq:sdf_void3d}
\end{equation}
which is negative precisely in the shell $R_2 < r < R_1$. The BFF
$\Phi_\beta$ follows from Eq.~\eqref{eq:Phi_def}. The exact solution,
\begin{equation}
    u^{\text{Ext}}(\mathbf{x}) = \bigl(R_1^2 - r^2\bigr)\bigl(r^2 - R_2^2\bigr),
    \label{eq:void3d_exact}
\end{equation}
vanishes on both the outer sphere and the inner void, enforcing the homogeneous
Dirichlet condition, and the source term $b=-\Delta u^{\text{Ext}}$ is obtained
analytically.


Figure~\ref{fig:void3d_conv} shows the relative $L^2$ and energy errors for
$s=p=1$ and $s=p=2$, with mesh size $h = 2/n_x$. The measured rates over the finest
meshes, summarized in Table~\ref{tab:void3d_conv_rates}, recover the expected optimal
orders $\mathcal{O}(h^{p+1})$ in $L^2$ and $\mathcal{O}(h^{p})$ in energy, confirming
that the body-fitted description retains full accuracy in three dimensions.

\begin{table}[htbp]
\centering
\caption{Estimated convergence rates for the 3D spherical shell (finest mesh refinements)}
\label{tab:void3d_conv_rates}
\begin{tabular}{lccc}
\hline
 & \textbf{$L^2$ rate} & \textbf{Energy rate} & \textbf{Theoretical optimum} \\
\hline
$s=p=1$ & $2.14$ & $1.10$ & $2.0$ / $1.0$ \\
$s=p=2$ & $2.89$ & $1.89$ & $3.0$ / $2.0$ \\
\hline
\end{tabular}
\end{table}

\begin{figure}[htbp]
\centering
\includegraphics[width=0.95\textwidth]{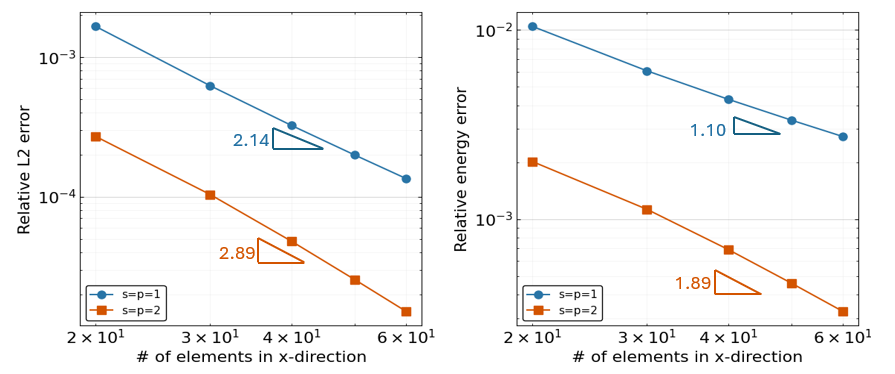}

\caption{Convergence of the relative $L^2$ error (left) and relative energy error
(right) for the 3D sphere with a concentric spherical void ($R_1=0.7$, $R_2=0.3$,
$\mathbf{x}_c=(1,1,1)$). The triangles indicate the fitted slopes.}
\label{fig:void3d_conv}
\end{figure}


\section{Discussions}
\label{sec:discussions}
The Immersed Tensor Decomposition (ITD) framework successfully integrates non-conforming geometries into reduced-order models by intertwining a continuous body-fitted representation with a separable C-HiDeNN-TD solver. This approach introduces several key advantages and practical considerations:

Strong Boundary Enforcement without Interface Quadrature: Unlike classical immersed methods that rely on weak Nitsche-type enforcement, ITD satisfies Dirichlet conditions strongly via the BFF $\Phi$. This eliminates penalty parameter tuning and complex interface quadratures, preserving the purely Cartesian tensor-product structure required for efficient reduced-order modeling.

Algebraic Complexity Reduction: A direct substitution of the body-fitted trial function into the weak form would generate six distinct geometric coefficient fields in~3D (five in~2D). Our exact algebraic cancellation reduces this to just three scalar fields ($\Phi$, $\Phi^2$, and $\Phi\,\nabla^2\Phi$). This reduction cuts the tensor compression overhead by nearly a half. 

Balancing Hyperparameters ($\beta$ and $\gamma$): Practical implementation requires balancing the exponential transition scaling $\beta$—which must smooth the SDF without demanding excessively dense quadrature—against the void stabilization $\gamma$. As established numerically and theoretically, $\gamma$ must scale as $\mathcal{O}(h^{2p})$ to effectively stabilize the void without polluting the solution in the physical domain.

\section{Conclusions and Future Work}\label{sec:CONCLUSION}
This paper presented the Immersed Tensor Decomposition (ITD) framework, which successfully integrates mesh-free geometric representations with the separable C-HiDeNN-TD reduced-order solver. By utilizing a BFF ($\Phi$) derived from a signed-distance field, the formulation achieves strong, exact enforcement of Dirichlet boundary conditions on regular Cartesian background grids. This approach circumvents the need for conforming body-fitted meshes, variational penalty parameters, and irregular interface quadratures.A central mathematical contribution of this work is the algebraically transformed weak form, which collapses the geometric complexity from six initial geometric coefficients in~3D (five in~2D) into just three volume fields ($\Phi$, $\Phi^2$, and $\Phi\,\nabla^2\Phi$). This reduction is critical for maintaining the efficiency of the low-rank Tucker decomposition and avoiding the curse of dimensionality. We established a rigorous a priori error estimate demonstrating that the ITD formulation preserves the optimal $\mathcal{O}(h^{p+1})$ convergence rate of the underlying high-order C-HiDeNN interpolation, provided the void stabilization parameter and the Tucker compression perturbations are properly scaled. Numerical evaluations on canonical 2D domains and 3D geometries confirmed these theoretical bounds, demonstrating robust sub-domain resolution and optimal energy norm convergence. Future research will focus on extending the ITD framework in the following directions: 

Vector-field and non-linear PDEs: While the current framework was derived and validated for the scalar Poisson equation, extending the formulation to coupled vector-field problems, such as linear elasticity, is necessary for broader computational mechanics and computer-aided engineering applications.

Direct ingestion of voxel-native data: We intend to deploy the ITD framework directly on raw volumetric datasets, such as medical imaging (CT/MRI) and additive manufacturing lattice structures. In these scenarios, the generation of signed-distance fields and subsequent tensor compression can completely bypass traditional, information-lossy CAD surface reconstruction workflows. 

Adaptive rank selection: The derived error estimates indicate that Tucker perturbation fields act as an effective additional stabilization. Developing an adaptive rank-selection algorithm to dynamically balance the tensor compression error against the intrinsic discretization error will further optimize offline pre-computation costs.

\appendix

\section{Proof of Theorem~\ref{thm:error}}
\label{appenx:Proof}

For simplicity, we consider homogeneous Dirichlet boundary condition ($g^h = 0$); the non‑homogeneous case follows by a standard lifting argument. 
The stabilized weak form can be written as
\begin{equation}\label{eq:general_weak}
    A_\gamma(W^h, u^h) = L(W^h, b),
\end{equation}
where the bilinear form $A_\gamma(\cdot,\cdot)$ and the linear functional $L(\cdot,b)$ in Eq. (\ref{eq:weak_strong}) are defined by
\begin{eqnarray}
    A_\gamma(W^h, u^h) &=& \int_{\mathcal{D}} \nabla W^h \cdot \nabla u^h \,\mathrm{d}\mathcal{D} + \int_{\mathcal{D}} \gamma(\mathbf{x}) \nabla (W^h/\Phi) \cdot \nabla (u^h/\Phi) \,\mathrm{d}\mathcal{D}, \\
    & = & \int_{\mathcal{D}} \bigl(\Phi^2 + \gamma(\mathbf{x})\bigr)\,(\nabla (W^h/\Phi) \cdot \nabla (u^h/\Phi)) \,\mathrm{d}\mathcal{D}
    \;-\; \int_{\mathcal{D}} (\nabla^2 \Phi / \Phi)\, W^h\, u^h \,\mathrm{d}\mathcal{D}, \\
    L(W^h, b) &=& \int_{\mathcal{D}} W^h\,b \,\mathrm{d}\mathcal{D}.
\end{eqnarray}
The test and trial spaces coincide:
\begin{equation}
    \mathcal{V}^h = \mathcal{S}^h = \bigl\{ W^h = \Phi w^h :\; w^h = \mathcal{I}^h w,\; w \in H^{p+1}(\mathcal{D}),\; w|_{\partial\mathcal{D}} = 0 \bigr\},
\end{equation}
where $\mathcal{I}^h$ denotes the C‑HiDeNN interpolation operator. 
The energy norm is $\|w\|_\gamma = \sqrt{A_\gamma(w,w)}$.

\begin{lemma}[Strang's first lemma] \label{lemma1}
    Let $u$ be the exact solution and $u^h$ the discrete solution satisfying $A_\gamma(W^h, u^h) = L(W^h, b)$ for all $W^h \in \mathcal{V}^h$. Then
    \begin{equation}\label{eq:strang}
        \|u - u^h\|_\gamma \;\leq\; 2 \inf_{z^h \in \mathcal{S}^h} \|u - z^h\|_\gamma \;+\; \sup_{\eta^h \in \mathcal{V}^h\setminus\{0\}} \frac{\bigl| A_\gamma(\eta^h, u) - L(\eta^h, b) \bigr|}{\|\eta^h\|_\gamma}.
    \end{equation}
\end{lemma}

\begin{proof}
    For any $z^h \in \mathcal{S}^h$, set $e^h = z^h - u^h$. Using the weak form,
    \begin{align*}
        \|e^h\|_\gamma^2 &= A_\gamma(e^h, z^h - u^h) = A_\gamma(e^h, z^h) - L(e^h, b) \\
                         &= A_\gamma(e^h, z^h - u) + A_\gamma(e^h, u) - L(e^h, b).
    \end{align*}
    By the Cauchy–Schwarz inequality,
    \[
        A_\gamma(e^h, z^h - u) \le \|e^h\|_\gamma \|z^h - u\|_\gamma.
    \]
    Hence,
    \[
        \|e^h\|_\gamma \le \|z^h - u\|_\gamma + \frac{|A_\gamma(e^h, u) - L(e^h, b)|}{\|e^h\|_\gamma}
                      \le \|z^h - u\|_\gamma + \sup_{\eta^h \in \mathcal{V}^h\setminus\{0\}} \frac{|A_\gamma(\eta^h, u) - L(\eta^h, b)|}{\|\eta^h\|_\gamma}.
    \]
    The triangle inequality $\|u - u^h\|_\gamma \le \|u - z^h\|_\gamma + \|z^h - u^h\|_\gamma$ yields
    \[
        \|u - u^h\|_\gamma \le 2\|u - z^h\|_\gamma + \sup_{\eta^h \in \mathcal{V}^h\setminus\{0\}} \frac{|A_\gamma(\eta^h, u) - L(\eta^h, b)|}{\|\eta^h\|_\gamma}.
    \]
    Taking the infimum over $z^h$ gives~\eqref{eq:strang}.
\end{proof}

\noindent\textbf{Proof of Theorem~\ref{thm:error}.}
Let the exact solution be extended to the whole background domain by $u = \Phi \hat{v}$, where $\hat{v}$ is the harmonic extension. 
Since $\hat{v}$ may have limited smoothness across $\partial\Omega$, we introduce a Stein extension $\tilde{v} \in H^{p+1}(\mathcal{D})$ with $\tilde{v}|_\Omega = v|_\Omega$.
Using~\eqref{eq:strang}, we estimate the two terms separately.

\paragraph{Approximation error}
\begin{align*}
    2\inf_{z^h \in \mathcal{S}^h} \|u - z^h\|_\gamma
    &\le 2 \|u - \Phi\tilde{v}\|_\gamma +2 \inf_{z^h} \|\Phi\tilde{v} - z^h\|_\gamma \\
    &= 2 \Bigl( \int_{\mathcal{D}\setminus\Omega} \gamma\, \bigl|\nabla(\hat{v}-\tilde{v})\bigr|^2 \,\mathrm{d}\mathcal{D} \Bigr)^{1/2} + C_1 h^p\|\tilde{v}\|_{H^{p+1}(\mathcal{D})} \\
    &\le C_3 \sqrt{\sup_{\mathbf{x}\in\mathcal{D}\setminus\Omega} \gamma(\mathbf{x})}\,\|\hat{v} - \tilde{v}\|_{H^1(\mathcal{D}\setminus\overline{\Omega})} + C_1 h^{p}\, \|\tilde{v}\|_{H^{p+1}(\mathcal{D})},
\end{align*}
where the first term reflects the mismatch between the harmonic and smooth extensions in the void, and the second term follows from the approximation properties of C‑HiDeNN.

\paragraph{Consistency error}
The exact solution satisfies the unstabilized weak form $A_0(W^h, u) = L(W^h, b)$. Consequently,
\[
    A_\gamma(\eta^h, u) - L(\eta^h, b) = \int_{\mathcal{D}} \gamma\, \nabla(\eta^h/\Phi) \cdot \nabla(\hat{v}) \,\mathrm{d}\mathcal{D}.
\]
As $\int_{\mathcal{D}\setminus\Omega} \nabla w \cdot \nabla(\hat{v}) \,\mathrm{d}\mathcal{D}$ for any weighting function $w\in H^1(\mathcal{D}\setminus\Omega)$ and $w|_{\partial \Omega} = 0 $, this residual reduces to
\begin{equation}
    A_\gamma(\eta^h, u) - L(\eta^h, b) = \int_{\Omega} \gamma\, \nabla(\eta^h/\Phi) \cdot \nabla(\hat{v}) \,\mathrm{d}\Omega.
\end{equation}
By the Cauchy–Schwarz inequality and the definition of the $\gamma$‑norm,
\begin{align*}
    \bigl| A_\gamma(\eta^h, u) - L(\eta^h, b) \bigr|
    &\le \sup_{\mathbf{x}\in\Omega \gamma(\mathbf{x})} \bigl\|  \nabla(\eta^h/\Phi) \bigr\|_{L^2(\Omega)} \,
       \bigl\| \nabla(\hat{v}) \bigr\|_{L^2(\Omega)}.
\end{align*}
Using the inequality
    \begin{equation}
        \|\eta^h\|^2_\gamma \geq  A_0 (\eta^h, \eta^h)=\|\nabla \eta^h\|^2_{L^2(\Omega)},
    \end{equation}
we obtain
\[
    \sup_{\eta^h \in \mathcal{V}^h\setminus\{0\}} \frac{|A_\gamma(\eta^h, u) - L(\eta^h, b)|}{\|\eta^h\|_\gamma}
    \le C_2 \sup_{\mathbf{x}\in\Omega} \gamma(\mathbf{x}) \sup_{\eta^h \in \mathcal{V}^h\setminus\{0\}} \dfrac{\|\nabla \eta^h/\Phi\|_{L^2(\Omega)}}{\|\nabla \eta^h\|_{L^2(\Omega)}} \cdot \|u\|_{H^1(\Omega)}.
\]

Combining the two bounds in~\eqref{eq:strang} gives
\[
    \|u - u^h\|_\gamma \le C_1 h^{p}\, \|\tilde{v}\|_{H^{p+1}(\mathcal{D})}
    + C_2 \sup_{\mathbf{x}\in\Omega} \gamma(\mathbf{x}) \sup_{\eta^h \in \mathcal{V}^h\setminus\{0\}} \dfrac{\|\nabla \eta^h/\Phi\|_{L^2(\Omega)}}{\|\nabla \eta^h\|_{L^2(\Omega)}} \cdot \|u\|_{H^1(\Omega)}
    + C_3 \sqrt{\sup_{\mathbf{x}\in\mathcal{D}\setminus\Omega} \gamma(\mathbf{x})}\, \|\hat{v} - \tilde{v}\|_{H^1(\mathcal{D}\setminus\overline{\Omega})},
\]
which is precisely the statement of Theorem~\ref{thm:error}.

\section{Proof of Theorem~\ref{thm:error_tucker}}
\label{appenx:Proof_tucker}

We follow the structure of the proof of Theorem~\ref{thm:error} (Appendix~\ref{appenx:Proof}), treating the Tucker perturbation fields $\delta_1, \delta_2, \delta_3$ as additional stabilization-like terms. For simplicity, we consider homogeneous Dirichlet boundary condition ($g^h = 0$); the non-homogeneous case follows by a standard lifting argument.


Substituting the Tucker-compressed geometric coefficients $\hat{\Phi}^{(1)}, \hat{\Phi}^{(2)}, \hat{\Phi}^{(3)}$ into the exact Dirichlet weak form~\eqref{eq:weak_strong} yields the Tucker-perturbed bilinear form and linear functional:
\begin{eqnarray}
    \hat{A}_\gamma(w, v) &=& \int_{\mathcal{D}} \bigl(\hat{\Phi}^{(2)} + \gamma(\mathbf{x})\bigr)\,(\nabla w \cdot \nabla v) \,\mathrm{d}\mathcal{D}
                           \;-\; \int_{\mathcal{D}} \hat{\Phi}^{(3)}\, w\, v \,\mathrm{d}\mathcal{D}, \label{eq:A_hat} \\
    \hat{L}(w, b) &=& \int_{\mathcal{D}} \hat{\Phi}^{(1)}\, w\, b \,\mathrm{d}\mathcal{D}. \label{eq:L_hat}
\end{eqnarray}
Using the perturbation fields $\delta_1 = \hat{\Phi}^{(1)} - \Phi$, $\delta_2 = \hat{\Phi}^{(2)} - \Phi^2$, $\delta_3 = \hat{\Phi}^{(3)} - \Phi\,\nabla^2\Phi$, the Tucker forms relate to the exact forms~\eqref{eq:general_weak} as
\begin{eqnarray}
    \hat{A}_\gamma(w, v) &=& A_\gamma(w, v) + \int_{\mathcal{D}} \delta_2\,(\nabla w \cdot \nabla v) \,\mathrm{d}\mathcal{D}
                           \;-\; \int_{\mathcal{D}} \delta_3\, w\, v \,\mathrm{d}\mathcal{D}, \label{eq:A_hat_decomp} \\
    \hat{L}(w, b) &=& L(w, b) + \int_{\mathcal{D}} \delta_1\, w\, b \,\mathrm{d}\mathcal{D}. \label{eq:L_hat_decomp}
\end{eqnarray}
The discrete Tucker solution $u^h$ satisfies $\hat{A}_\gamma(\eta^h, u^h) = \hat{L}(\eta^h, b)$ for all $\eta^h \in \mathcal{V}^h$, while the exact solution $u$ satisfies $A_0(\eta, u) = L(\eta, b)$ for all $\eta \in \mathcal{V}^h$.


Define the Tucker energy norm $\|w\|_{\hat{\gamma}} = \sqrt{\hat{A}_\gamma(w,w)}$. From~\eqref{eq:A_hat_decomp},
\begin{equation}
    \bigl| \|w\|_{\hat{\gamma}}^2 - \|w\|_\gamma^2 \bigr|
    \;\leq\; \|\delta_2\|_{L^\infty(\mathcal{D})} \|\nabla w\|^2_{L^2(\mathcal{D})}
    \;+\; \|\delta_3\|_{L^\infty(\mathcal{D})} \|w\|^2_{L^2(\mathcal{D})}. \label{eq:norm_diff}
\end{equation}
Since $\Phi \in [0,1]$ and $w = \Phi\,\tilde{w}$ for some $\tilde{w} \in H^1(\mathcal{D})$, the $L^2$ term is controlled by the gradient term via the Poincar\'{e}--Friedrichs inequality. Assuming the Tucker compression is sufficiently accurate that $\|\delta_2\|_{L^\infty}$ and $\|\delta_3\|_{L^\infty}$ are small, the two norms are equivalent: there exist constants $c_0, c_1 > 0$ independent of $h$ and of the Tucker rank such that
\begin{equation}
    c_0 \|w\|_\gamma \;\leq\; \|w\|_{\hat{\gamma}} \;\leq\; c_1 \|w\|_\gamma, \qquad \forall\, w \in \mathcal{V}^h. \label{eq:norm_equiv}
\end{equation}
In the estimates below, we absorb $c_0^{-1}$ and $c_1$ into the generic constants $C_2, \dots, C_5$.


Applying Strang's first lemma (Lemma~\ref{lemma1}) to the Tucker form $\hat{A}_\gamma$:
\begin{equation}\label{eq:strang_tucker}
    \|u - u^h\|_{\hat{\gamma}} \;\leq\; 2 \inf_{z^h \in \mathcal{S}^h} \|u - z^h\|_{\hat{\gamma}}
    \;+\; \sup_{\eta^h \in \mathcal{V}^h\setminus\{0\}} \frac{\bigl| \hat{A}_\gamma(\eta^h, u) - \hat{L}(\eta^h, b) \bigr|}{\|\eta^h\|_{\hat{\gamma}}}.
\end{equation}
We estimate the approximation error and the consistency error separately.


Let $u = \Phi \hat{v}$ be the exact solution extended harmonically to $\mathcal{D}$, and let $\tilde{v} \in H^{p+1}(\mathcal{D})$ be the Stein extension. Decompose the approximation error:
\begin{align*}
    2\inf_{z^h \in \mathcal{S}^h} \|u - z^h\|_{\hat{\gamma}}
    &\le 2 \|u - \Phi\tilde{v}\|_{\hat{\gamma}} + 2 \inf_{z^h} \|\Phi\tilde{v} - z^h\|_{\hat{\gamma}}.
\end{align*}
For the first term, using the definition of $\hat{A}_\gamma$ and noting that $u - \Phi\tilde{v} = \Phi(\hat{v} - \tilde{v})$ vanishes on $\Omega$:
\begin{align*}
    \|u - \Phi\tilde{v}\|_{\hat{\gamma}}^2
    &= \int_{\mathcal{D}\setminus\Omega} \bigl(\Phi^2 + \delta_2 + \gamma\bigr) \bigl|\nabla(\hat{v} - \tilde{v})\bigr|^2 \,\mathrm{d}\mathcal{D}
       \;-\; \int_{\mathcal{D}\setminus\Omega} (\Phi\,\nabla^2\Phi + \delta_3)\, |\hat{v} - \tilde{v}|^2 \,\mathrm{d}\mathcal{D} \\
    &\le \bigl(\sup_{\mathbf{x}\in\mathcal{D}\setminus\Omega} \gamma(\mathbf{x}) + \|\delta_2\|_{L^\infty(\mathcal{D}\setminus\Omega)}\bigr) \|\hat{v} - \tilde{v}\|^2_{H^1(\mathcal{D}\setminus\overline{\Omega})}
       \;+\; \|\delta_3\|_{L^\infty(\mathcal{D})} \|\hat{v} - \tilde{v}\|^2_{L^2(\mathcal{D})} \\
    &\le \bigl(\sup_{\mathbf{x}\in\mathcal{D}\setminus\Omega} \gamma(\mathbf{x}) + \|\delta_2\|_{L^\infty(\mathcal{D}\setminus\Omega)} + \|\delta_3\|_{L^\infty(\mathcal{D})}\bigr) \|\hat{v} - \tilde{v}\|^2_{H^1(\mathcal{D}\setminus\overline{\Omega})}.
\end{align*}

For the second term, the C-HiDeNN interpolation estimate in the $\hat{\gamma}$-norm is identical to that in Theorem~\ref{thm:error} up to the perturbation of the norm, which is controlled by~\eqref{eq:norm_equiv}. Hence,
\begin{equation}
    2\inf_{z^h} \|\Phi\tilde{v} - z^h\|_{\hat{\gamma}} \;\leq\; C_1 h^{p}\, \|\tilde{v}\|_{H^{p+1}(\mathcal{D})}.
\end{equation}


Decompose the consistency residual using~\eqref{eq:A_hat_decomp}--\eqref{eq:L_hat_decomp}:
\begin{align}
    \hat{A}_\gamma(\eta^h, u) - \hat{L}(\eta^h, b)
    &= \bigl[A_\gamma(\eta^h, u) - L(\eta^h, b)\bigr] \label{eq:cons_decomp_1} \\
    &\quad + \int_{\mathcal{D}} \delta_2\, \nabla w^h \cdot \nabla \hat{v} \,\mathrm{d}\mathcal{D} \label{eq:cons_decomp_2} \\
    &\quad - \int_{\mathcal{D}} \delta_3\, w^h\, \hat{v} \,\mathrm{d}\mathcal{D} \label{eq:cons_decomp_3} \\
    &\quad - \int_{\mathcal{D}} \delta_1\, w^h\, \tilde{b} \,\mathrm{d}\mathcal{D}, \label{eq:cons_decomp_4}
\end{align}
where $w^h = \eta^h / \Phi$, $\hat{v} = u / \Phi$, and $\tilde{b} = \nabla^2 g^h + b$ (with $b$ the body force).

\paragraph{Term~\eqref{eq:cons_decomp_1}: $\gamma$-stabilization}
From the proof of Theorem~\ref{thm:error} (~\ref{appenx:Proof}), the exact stabilized consistency error satisfies
\begin{equation}
    \bigl| A_\gamma(\eta^h, u) - L(\eta^h, b) \bigr|
    \;\leq\; \sup_{\mathbf{x}\in\Omega} \gamma(\mathbf{x})\, \|\nabla w^h\|_{L^2(\Omega)}\, \|\nabla \hat{v}\|_{L^2(\Omega)}.
\end{equation}
Dividing by $\|\eta^h\|_{\hat{\gamma}}$ and using the norm equivalence~\eqref{eq:norm_equiv} together with the inequality $\|\eta^h\|_\gamma \ge \|\nabla \eta^h\|_{L^2(\Omega)}$ proved in Lemma~\ref{lemma1} yields the bound with $\sup_{\Omega} \gamma(\mathbf{x})$ as in the second term of Eq. ~\eqref{eq:error_estimate_tucker}.

\paragraph{Term~\eqref{eq:cons_decomp_2}: $\delta_2$ stiffness perturbation}
The perturbation $\delta_2$ enters as a coefficient of the stiffness integral, exactly as $\gamma$ does. Splitting the integral over $\Omega$ and $\mathcal{D}\setminus\Omega$:
\[
    \int_{\mathcal{D}} \delta_2\, \nabla w^h \cdot \nabla \hat{v} \,\mathrm{d}\mathcal{D}
    \;=\; \int_{\Omega} \delta_2\, \nabla w^h \cdot \nabla \hat{v} \,\mathrm{d}\Omega
      \;+\; \int_{\mathcal{D}\setminus\Omega} \delta_2\, \nabla w^h \cdot \nabla \hat{v} \,\mathrm{d}(\mathcal{D}\setminus\Omega).
\]
In the interior $\Omega$, the Cauchy--Schwarz inequality gives
\[
    \biggl| \int_{\Omega} \delta_2\, \nabla w^h \cdot \nabla \hat{v} \,\mathrm{d}\Omega \biggr|
    \;\leq\; \|\delta_2\|_{L^\infty(\Omega)}\, \|\nabla w^h\|_{L^2(\Omega)}\, \|\nabla \hat{v}\|_{L^2(\Omega)}.
\]
In the void $\mathcal{D}\setminus\Omega$, the harmonic extension satisfies $\nabla^2 \hat{v} = 0$. Integrating by parts and using $w^h|_{\partial\Omega} = 0$ (since $w^h = \eta^h/\Phi$ and $\eta^h|_{\partial\mathcal{D}} = 0$), the void integral vanishes under the same argument as the $\gamma$ term in Theorem~\ref{thm:error}:
\[
    \int_{\mathcal{D}\setminus\Omega} \delta_2\, \nabla w^h \cdot \nabla \hat{v} \,\mathrm{d}(\mathcal{D}\setminus\Omega) = 0.
\]
Thus, the $\delta_2$ contribution combines additively with $\gamma$ in the interior bound, yielding the term $\|\delta_2\|_{L^\infty(\Omega)}$ inside the second line of~\eqref{eq:error_estimate_tucker}.

\paragraph{Term~\eqref{eq:cons_decomp_3}: $\delta_3$ mass perturbation}
By the Cauchy--Schwarz inequality,
\begin{align*}
    \biggl| \int_{\mathcal{D}} \delta_3\, w^h\, \hat{v} \,\mathrm{d}\mathcal{D} \biggr|
    &\;\leq\; \|\delta_3\|_{L^\infty(\mathcal{D})}\, \|w^h\|_{L^2(\mathcal{D})}\, \|\hat{v}\|_{L^2(\mathcal{D})} \\
    &\;\leq\; \|\delta_3\|_{L^\infty(\mathcal{D})}\, \|\eta^h/\Phi\|_{L^2(\mathcal{D})}\, \|\hat{v}\|_{L^2(\mathcal{D})}.
\end{align*}
Dividing by $\|\eta^h\|_{\hat{\gamma}}$, taking the supremum over $\eta^h$, we obtain the fourth term of~\eqref{eq:error_estimate_tucker}.

\paragraph{Term~\eqref{eq:cons_decomp_4}: $\delta_1$ right-hand side perturbation.}
\begin{align*}
    \biggl| \int_{\mathcal{D}} \delta_1\, w^h\, \tilde{b} \,\mathrm{d}\mathcal{D} \biggr|
    &\;\leq\; \|\delta_1\|_{L^\infty(\mathcal{D})}\, \|w^h\|_{L^2(\mathcal{D})}\, \|\tilde{b}\|_{L^2(\mathcal{D})} \\
    &\;\leq\; \|\delta_1\|_{L^\infty(\mathcal{D})}\, \|\eta^h/\Phi\|_{L^2(\mathcal{D})}\, \|\tilde{b}\|_{L^2(\mathcal{D})}.
\end{align*}


Combining the approximation error and the four consistency error contributions into the Strang decomposition~\eqref{eq:strang_tucker} and converting from $\|\cdot\|_{\hat{\gamma}}$ to $\|\cdot\|_\gamma$ via~\eqref{eq:norm_equiv}, we obtain the error estimate given in Eq. ~\eqref{eq:error_estimate_tucker}.
\hfill $\square$

\bibliography{reference}


\end{document}